\documentclass{article} 

\usepackage{amsmath}
\usepackage{amssymb}
\usepackage{amstext}
\usepackage{xspace}
\usepackage{url}
\usepackage{xy}
\xyoption{all}
\usepackage{stmaryrd}

\newdimen\proofrulebreadth \proofrulebreadth=.05em
\newdimen\proofdotseparation \proofdotseparation=1.25ex
\newdimen\proofrulebaseline \proofrulebaseline=2ex
\newcount\proofdotnumber \proofdotnumber=3
\let\then\relax
\def\hfi{\hskip0pt plus.0001fil}
\mathchardef\squigto="3A3B
\newif\ifinsideprooftree\insideprooftreefalse
\newif\ifonleftofproofrule\onleftofproofrulefalse
\newif\ifproofdots\proofdotsfalse
\newif\ifdoubleproof\doubleprooffalse
\let\wereinproofbit\relax
\newdimen\shortenproofleft
\newdimen\shortenproofright
\newdimen\proofbelowshift
\newbox\proofabove
\newbox\proofbelow
\newbox\proofrulename
\def\shiftproofbelow{\let\next\relax\afterassignment\setshiftproofbelow\dimen0 }
\def\shiftproofbelowneg{\def\next{\multiply\dimen0 by-1 }%
\afterassignment\setshiftproofbelow\dimen0 }
\def\setshiftproofbelow{\next\proofbelowshift=\dimen0 }
\def\setproofrulebreadth{\proofrulebreadth}

\def\prooftree{
\ifnum  \lastpenalty=1
\then   \unpenalty
\else   \onleftofproofrulefalse
\fi
\ifonleftofproofrule
\else   \ifinsideprooftree
        \then   \hskip.5em plus1fil
        \fi
\fi
\bgroup
\setbox\proofbelow=\hbox{}\setbox\proofrulename=\hbox{}%
\let\justifies\proofover\let\leadsto\proofoverdots\let\Justifies\proofoverdbl
\let\using\proofusing\let\[\prooftree
\ifinsideprooftree\let\]\endprooftree\fi
\proofdotsfalse\doubleprooffalse
\let\thickness\setproofrulebreadth
\let\shiftright\shiftproofbelow \let\shift\shiftproofbelow
\let\shiftleft\shiftproofbelowneg
\let\ifwasinsideprooftree\ifinsideprooftree
\insideprooftreetrue
\setbox\proofabove=\hbox\bgroup$\displaystyle 
\let\wereinproofbit\prooftree
\shortenproofleft=0pt \shortenproofright=0pt \proofbelowshift=0pt
\onleftofproofruletrue\penalty1
}

\def\eproofbit{
\ifx    \wereinproofbit\prooftree
\then   \ifcase \lastpenalty
        \then   \shortenproofright=0pt  
        \or     \unpenalty\hfil         
        \or     \unpenalty\unskip       
        \else   \shortenproofright=0pt  
        \fi
\fi
\global\dimen0=\shortenproofleft
\global\dimen1=\shortenproofright
\global\dimen2=\proofrulebreadth
\global\dimen3=\proofbelowshift
\global\dimen4=\proofdotseparation
\global\count255=\proofdotnumber
$\egroup  
\shortenproofleft=\dimen0
\shortenproofright=\dimen1
\proofrulebreadth=\dimen2
\proofbelowshift=\dimen3
\proofdotseparation=\dimen4
\proofdotnumber=\count255
}

\def\proofover{
\eproofbit 
\setbox\proofbelow=\hbox\bgroup 
\let\wereinproofbit\proofover
$\displaystyle
}%
\def\proofoverdbl{
\eproofbit 
\doubleprooftrue
\setbox\proofbelow=\hbox\bgroup 
\let\wereinproofbit\proofoverdbl
$\displaystyle
}%
\def\proofoverdots{
\eproofbit 
\proofdotstrue
\setbox\proofbelow=\hbox\bgroup 
\let\wereinproofbit\proofoverdots
$\displaystyle
}%
\def\proofusing{
\eproofbit 
\setbox\proofrulename=\hbox\bgroup 
\let\wereinproofbit\proofusing
\kern0.3em$
}

\def\endprooftree{
\eproofbit 
  \dimen5 =0pt
\dimen0=\wd\proofabove \advance\dimen0-\shortenproofleft
\advance\dimen0-\shortenproofright
\dimen1=.5\dimen0 \advance\dimen1-.5\wd\proofbelow
\dimen4=\dimen1
\advance\dimen1\proofbelowshift \advance\dimen4-\proofbelowshift
\ifdim  \dimen1<0pt
\then   \advance\shortenproofleft\dimen1
        \advance\dimen0-\dimen1
        \dimen1=0pt
        \ifdim  \shortenproofleft<0pt
        \then   \setbox\proofabove=\hbox{%
                        \kern-\shortenproofleft\unhbox\proofabove}%
                \shortenproofleft=0pt
        \fi
\fi
\ifdim  \dimen4<0pt
\then   \advance\shortenproofright\dimen4
        \advance\dimen0-\dimen4
        \dimen4=0pt
\fi
\ifdim  \shortenproofright<\wd\proofrulename
\then   \shortenproofright=\wd\proofrulename
\fi
\dimen2=\shortenproofleft \advance\dimen2 by\dimen1
\dimen3=\shortenproofright\advance\dimen3 by\dimen4
\ifproofdots
\then
        \dimen6=\shortenproofleft \advance\dimen6 .5\dimen0
        \setbox1=\vbox to\proofdotseparation{\vss\hbox{$\cdot$}\vss}%
        \setbox0=\hbox{%
                \advance\dimen6-.5\wd1
                \kern\dimen6
                $\vcenter to\proofdotnumber\proofdotseparation
                        {\leaders\box1\vfill}$%
                \unhbox\proofrulename}%
\else   \dimen6=\fontdimen22\the\textfont2 
        \dimen7=\dimen6
        \advance\dimen6by.5\proofrulebreadth
        \advance\dimen7by-.5\proofrulebreadth
        \setbox0=\hbox{%
                \kern\shortenproofleft
                \ifdoubleproof
                \then   \hbox to\dimen0{%
                        $\mathsurround0pt\mathord=\mkern-6mu%
                        \cleaders\hbox{$\mkern-2mu=\mkern-2mu$}\hfill
                        \mkern-6mu\mathord=$}%
                \else   \vrule height\dimen6 depth-\dimen7 width\dimen0
                \fi
                \unhbox\proofrulename}%
        \ht0=\dimen6 \dp0=-\dimen7
\fi
\let\doll\relax
\ifwasinsideprooftree
\then   \let\VBOX\vbox
\else   \ifmmode\else$\let\doll=$\fi
        \let\VBOX\vcenter
\fi
\VBOX   {\baselineskip\proofrulebaseline \lineskip.2ex
        \expandafter\lineskiplimit\ifproofdots0ex\else-0.6ex\fi
        \hbox   spread\dimen5   {\hfi\unhbox\proofabove\hfi}%
        \hbox{\box0}%
        \hbox   {\kern\dimen2 \box\proofbelow}}\doll%
\global\dimen2=\dimen2
\global\dimen3=\dimen3
\egroup 
\ifonleftofproofrule
\then   \shortenproofleft=\dimen2
\fi
\shortenproofright=\dimen3
\onleftofproofrulefalse
\ifinsideprooftree
\then   \hskip.5em plus 1fil \penalty2
\fi
}

\makeatother
 \newdir{ >}{{}*!/-7.5pt/@{>}}
 \newdir{|>}{!/4.5pt/@{|}*:(1,-.2)@^{>}*:(1,+.2)@_{>}}
 \newdir{ |>}{{}*!/-3pt/@{|}*!/-7.5pt/:(1,-.2)@^{>}*!/-7.5pt/:(1,+.2)@_{>}}
\newcommand{\xyline}[2][]{\ensuremath{\smash{\xymatrix@1#1{#2}}}}
\newcommand{\xyinline}[2][]{\ensuremath{\smash{\xymatrix@1#1{#2}}}^{\rule[8.5pt]{0pt}{0pt}}}
\makeatletter

\newif\ifignore 
\ignorefalse
\newcommand{\auxproof}[1]{
\ifignore\mbox{}\newline
\textbf{PROOF:} \dotfill\newline
{\it #1}\mbox{}\newline
\textbf{ENDPROOF}\dotfill
\fi}

\newtheorem{theorem}{Theorem}
\newtheorem{lemma}[theorem]{Lemma}
\newtheorem{proposition}[theorem]{Proposition}

\newtheorem{definition}[theorem]{Definition}
\newtheorem{example}[theorem]{Example}

\newenvironment{proof}[1][Proof]%
   { \begin{trivlist}%
     \item[\hskip \labelsep {\bfseries #1}]%
   }%
   { \end{trivlist}%
   }
\newcommand{\QEDbox}{\square}
\newcommand{\QED}{\hspace*{\fill}$\QEDbox$}

\newcommand{\after}{\mathrel{\circ}}
\newcommand{\cat}[1]{\ensuremath{\mathbf{#1}}}
\newcommand{\Cat}[1]{\ensuremath{\mathbf{#1}}}

\newcommand{\op}{\ensuremath{^{\mathrm{op}}}}
\newcommand{\idmap}[1][]{\ensuremath{\mathrm{id}_{#1}}}

\newcommand{\support}{\ensuremath{\mathrm{supp}}}

\newcommand{\st}{\ensuremath{\mathsf{st}}}

\newcommand{\Hom}{\textsl{Hom}}
\newcommand{\PrimeFilter}{\textsl{pFil}}
\newcommand{\inprod}[2]{\ensuremath{\langle #1\,|\,#2 \rangle}}

\newcommand{\KSub}{\ensuremath{\mathrm{KSub}}}

\newcommand{\Alg}{\textsl{Alg}\xspace}
\newcommand{\Mlt}{\ensuremath{\mathcal{M}}}
\newcommand{\Dstr}{\ensuremath{\mathcal{D}}}

\newcommand{\NNO}{\ensuremath{\mathbb{N}}}
\newcommand{\unitR}{\ensuremath{[0,1]_{\mathbb{R}}}}
\newcommand{\nonnegR}{\ensuremath{\mathbb{R}_{\geq 0}}}

\newcommand{\orthogonal}{\mathrel{\bot}}

\newcommand{\scalar}{\mathrel{\bullet}}

\newcommand{\powerset}{\mathcal{P}}
\newcommand{\powersetfin}{\mathcal{P}_{\mathit{fin}}}
\newcommand{\powersetnefin}{\mathcal{P}_{\mathit{fin}}^{+}}

\newcommand{\leftScottint}{[{\kern-.3ex}[}
\newcommand{\rightScottint}{]{\kern-.3ex}]}

\newcommand{\Hilb}{\Cat{Hilb}\xspace}

\newcommand{\Sets}{\Cat{Sets}\xspace}

\newcommand{\cotuple}[1]{\ensuremath{[ #1 ]}}
\newcommand{\tuple}[1]{\ensuremath{\langle #1 \rangle}}

\newcommand{\set}[2]{\{#1\;|\;#2\}}
\newcommand{\setin}[3]{\{#1\in#2\;|\;#3\}}
\newcommand{\conjun}{\mathrel{\wedge}}
\newcommand{\disjun}{\mathrel{\vee}}
\newcommand{\dirjoin}{\bigvee\nolimits^{\uparrow}}
\newcommand{\all}[2]{\forall_{#1}.\,#2}

\newcommand{\ex}[2]{\exists_{#1}.\,#2}

\newcommand{\lam}[2]{\lambda{#1}.\,#2}

\newcommand{\EndoHom}[1]{{\cal E}{\kern-.5ex}\textit{n}{\kern-.2ex}\textit{d}{\kern-.2ex}\textit{o}(#1)}

\title{Duality for Convexity}
\author{Bart Jacobs, \\
{\small Institute for Computing and Information Sciences (iCIS),} \\[-.5em]
{\small Radboud University Nijmegen, The Netherlands.} \\[-.5em]
{\small Webaddress: \url{www.cs.ru.nl/B.Jacobs}}}

\date{\small \today}

\begin{document}
\maketitle

\begin{abstract}
  This paper studies convex sets categorically, namely as algebras of
  a distribution monad. It is shown that convex sets occur in two dual
  adjunctions, namely one with preframes via the Boolean truth values
  $\{0,1\}$ as dualising object, and one with effect algebras via the
  (real) unit interval $\unitR$ as dualising object. These effect
  algebras are of interest in the foundations of quantum mechanics.
\end{abstract}

\section{Introduction}\label{IntroSec}

A set $X$ is commonly called convex if for each pair of elements
$x,y\in X$ and each number $r\in\unitR$ in the unit interval of real
numbers the ``convex'' sum $rx + (1-r)y$ is again in $X$. Informally
this says that a whole line segment is contained in $X$ as soon as the
endpoints are in $X$.  Convexity is of course a well-established
notion that finds applications in for instance geometry, probability
theory, optimisation, economics and quantum mechanics (with mixed
states as convex combinations of pure states). The definition of
convexity (as just given) assumes a monoidal structure $+$ on the set
$X$ and also a scalar multiplication $\unitR\times X\rightarrow
X$. People have tried to capture this notion of convexity with fewer
assumptions, see for instance~\cite{NeumannM44}, \cite{Stone49}
or~\cite{Gudder79}. We shall use the latter source that involves a
ternary operation $\tuple{-,-,-} \colon \unitR\times X\times
X\rightarrow X$ satisfying a couple of equations, see
Definition~\ref{ConvexSetDef}. We first recall (see
\textit{e.g.}~\cite{Swirszcz74,Flood81,Keimel08,Doberkat06}) that such
convex structures can equivalently be described uniformly as algebras
of a monad, namely of the distribution monad $\Dstr$, see
Theorem~\ref{ConvexSetAlgThm}. Such an algebra map gives an
interpretation of each convex combination $r_{1}x_{1} + \cdots +
r_{n}x_{n}$, where $r_{1}+\cdots + r_{n}=1$, as a single element of
$X$. This algebraic description of convexity allows us to generalise
it from scalars $\unitR$ (or actually $\mathbb{R}_{\geq0}$) to
arbitrary semirings (or semifields) $S$ as scalars, and yields an
abstract description of a familiar embedding construction as an
adjunction between $S$-convex sets and $S$-semimodules, see
Proposition~\ref{DstrMltAlgAdjProp} below.

The main topic of this paper is duality for convex spaces. We shall
describe two dual adjunctions:
\begin{equation}
\label{DualAdjsDiag}
\xymatrix{
\Cat{PreFrm}\ar@/^2ex/ [rr]^{\Hom(-,\{0,1\})} & \bot & 
\Cat{Conv}\op \ar@/^2ex/ [ll]^{\Hom(-,\{0,1\})}
   \ar@/_2ex/[rr]_{\Hom(-,\unitR)} & \bot & 
   \Cat{EA}\ar@/_2ex/[ll]_{\Hom(-,\unitR)}
}
\end{equation}

\noindent namely in Theorems~\ref{ConvPreFrmDualAdjThm}
and~\ref{ConvEffAlgDualAdjThm}. This diagram involves the following
structures.
\begin{itemize}
\item The category \Cat{Conv} of (real) convex sets, with as special
  objects the unit interval $\unitR$ and the two element set
  $\{0,1\}$.  This unit interval captures probabilities, and $\{0,1\}$
  the Boolean truth values.

\item The category \Cat{PreFrm} of preframes: posets with directed
  joins and finite meets, distributing over these joins,
  see~\cite{JohnstoneV91}. These preframes are slightly more general
  than frames (or complete Heyting algebras) that occur in the
  familiar duality with topological spaces, see~\cite{Johnstone82}.

\item The category \Cat{EA} of effect algebras (from~\cite{FoulisB94},
  see also~\cite{DvurecenskijP00} for an overview). Effect algebras
  have arisen in the foundations of quantum mechanics and are used to
  capture quantum effects, as studied in quantum statistics and
  quantum measurement theory, see \textit{e.g.}~\cite{BuschGL95}.
\end{itemize}

\noindent The diagram~\eqref{DualAdjsDiag} thus suggests that convex
sets form a setting in which one can study both Boolean and
probabilistic logics. It opens up new questions, like: can the
adjunctions be refined further so that one actually obtains
equivalences, like between Stone spaces and Boolean algebras
(see~\cite{Johnstone82} for an overview). This is left to future
work. Dualities are important in algebra, topology and logic, for
transferring results and techniques from one domain to another.  They
are used in the semantics of computation (see
\textit{e.g.}~\cite{Abramsky91,Vickers89}), but are relatively new in
a quantum setting. They may become part of what is called
in~\cite{BaezS09} an ``extensive network of interlocking analogies
between physics, topology, logic and computer science''.

The paper starts with a preliminary section that recalls basic
definitions and facts about monads and their algebras. It leads to an
adjunction in Proposition~\ref{DstrMltAlgAdjProp} between two
categories of algebras, namely of the multiset monad and the
distribution monad. Section~\ref{ConvexSetSec} recalls in
Theorem~\ref{ConvexSetAlgThm} how (real) convex sets can be described
as algebras of the distribution monad, giving us the freedom to
generalise convexity to arbitrary semirings (as
scalars). Subsequently, Section~\ref{FilterSec} describes the
adjunction on the left in~\eqref{DualAdjsDiag} between convex sets and
preframes, via prime filters in convex sets and Scott-open filters in
preframes.  Both can be described via homomorphisms to the dualising
object $\{0,1\}$.  The adjunction on the right in~\eqref{DualAdjsDiag}
requires that we first sketch the basics of effect algebras. This is
done in Section~\ref{EffAlgSec}. The unit interval $\unitR$ now serves
as dualising object, where we note that effect algebra maps
$E\rightarrow\unitR$ are commonly studied as states in a quantum
system. The paper concludes in Section~\ref{HilbertSec} with a few
remarks about Hilbert spaces in relation to the dual
adjunctions~\eqref{DualAdjsDiag}.

\section{Monads and their algebras}\label{MonadSec}

A monad is a key concept in the generic categorical description of
algebraic structures. It is defined as a functor $T\colon \cat{C}
\rightarrow \cat{C}$ from a category \cat{C} to itself together with
two natural transformations, called the ``unit'' and
``multiplication'', see below. In the present context we don't need
the full generality and shall thus restrict ourselves to the case
where \cat{C} is the category \Sets of sets and functions. Associated
with a monad there is a category $\Alg(T)$ of algebras. Many
mathematical structures of interest arise in this uniform
manner. Categories of algebras $\Alg(T)$ satisfy certain useful
properties by default, see Theorem~\ref{MonadAlgStructThm}. This
section will review some standard definitions and results on monads
and their algebras that will be usefull in the sequel. More
information may be found in for
instance~\cite{MacLane71,BarrW85,Manes74}.

\begin{definition}
\label{MonadDef}
A monad (on \Sets) consists of an endofunctor
$T\colon\Sets\rightarrow\Sets$ together with two natural
transformations: a unit $\eta\colon \idmap \Rightarrow T$ and
multiplication $\mu \colon T^{2} \Rightarrow T$. These are required to
make the following diagrams commute, for $X\in\Sets$.
$$\xymatrix@C-1pc@R-.5pc{
T(X)\ar[rr]^-{\eta_{T(X)}}\ar@{=}[drr] & & T^{2}(X)\ar[d]^{\mu_X} & &
   T(X)\ar[ll]_{T(\eta_{X})}\ar@{=}[dll]
&  & T^{3}(X)\ar[rr]^-{\mu_{T(X)}}\ar[d]_{T(\mu_{X})} 
   & & T^{2}(X)\ar[d]^{\mu_X} \\
& & T(X) & &
& &
T^{2}\ar[rr]_{\mu_X} & & T(X)
}$$
\end{definition}

We mention a few instances of this definition, of which the last one
(the distribution monad $\Dstr$) will be most important.

\begin{example}
\label{MonadEx}
{\em(1)}~Let $(M,+,0)$ be a monoid. It can be used to construct a monad
$\widehat{M} \colon \Sets\rightarrow\Sets$ given by $\widehat{M}(X) =
M\times X$. The unit $\eta\colon X\rightarrow M\times X$ is $\eta(x) =
(0,x)$ and the multiplication $\mu \colon \widehat{M}^{2}(X) = M\times
(M\times X) \rightarrow M\times X = \widehat{M}(X)$ is given by
$\mu(a,(b,x)) = (a+b,x)$.

{\em(2)}~The powerset operation $\powerset$ forms a functor $\powerset
\colon \Sets \rightarrow \Sets$, which on a function $f\colon
X\rightarrow Y$ yields $\powerset(f) \colon \powerset(X)\rightarrow
\powerset(Y)$ by direct image: $\powerset(f)(U\subseteq X) =
\set{f(x)}{x\in U}$. Powerset is also a monad: the unit $\eta\colon
X\rightarrow \powerset(X)$ is given by singleton $\eta(x) = \{x\}$ and
multiplication $\mu\colon \powerset^{2}(X)\rightarrow \powerset(X)$ by
union $\mu(V\subseteq \powerset(X)) = \bigcup V$.

{\em(3)}~Let $S$ be a semiring, consisting of an additive monoid
$(S,+,0)$ and a multiplicative monoid $(S,\cdot,1)$, where
multiplication distributes over addition. One can define a
``multiset'' functor $\Mlt_{S}\colon\Sets\rightarrow\Sets$ by:
$$\begin{array}{rcl}
\Mlt_{S}(X)
& = &
\set{\varphi\colon X\rightarrow S}{\support(\varphi)\mbox{ is finite}},
\end{array}$$

\noindent where $\support(\varphi) = \setin{x}{X}{\varphi(x) \neq 0}$
is the support of $\varphi$. For a function $f\colon X\rightarrow Y$
one defines $\Mlt_{S}(f) \colon \Mlt_{S}(X) \rightarrow \Mlt_{S}(Y)$ by:
$$\begin{array}{rcl}
\Mlt_{S}(f)(\varphi)(y)
& = &
\sum_{x\in f^{-1}(y)}\varphi(x).
\end{array}$$

\noindent Such a multiset $\varphi\in \Mlt_{s}(X)$ may be written as
formal sum $s_{1}x_{1}+\cdots+s_{k}x_{k}$ where $\support(\varphi) =
\{x_{1}, \ldots, x_{k}\}$ and $s_{i} = \varphi(x_{i})\in S$ describes
the ``multiplicity'' of the element $x_{i}$. This formal sum notation
might suggest an order $1,2,\ldots k$ among the summands, but this is
misleading. The sum is considered, up-to-permutation of the summands.
Also, the same element $x\in X$ may be counted multiple times, but
$s_{1}x + s_{2}x$ is considered to be the same as $(s_{1}+s_{2})x$
within such expressions. With this formal sum notation one can write
the application of $\Mlt_{S}$ on a map $f$ as
$\Mlt_{S}(f)(\sum_{i}s_{i}x_{i}) =
\sum_{i}s_{i}f(x_{i})$. Functoriality is then obvious.

This multiset functor is a monad, with unit $\eta\colon X\rightarrow
\Mlt_{S}(X)$ is $\eta(x) = 1x$, and multiplication $\mu\colon
\Mlt_{S}(\Mlt_{S}(X)) \rightarrow \Mlt_{S}(X)$ given by
$\mu(\sum_{i}s_{i}\varphi_{i}) =
\lam{x}{\sum_{i}s_{i}\cdot\varphi_{i}(x)}$, where the ``lambda''
notation $\lam{x}{\cdots}$ is used for the function $x\mapsto \cdots$.

For the semiring $S=\NNO$ of natural numbers one gets the free
commutative monoid $\Mlt_{\NNO}(X)$ on a set $X$. And if
$S=\mathbb{Z}$ one obtains the free Abelian group
$\Mlt_{\mathbb{Z}}(X)$ on $X$. The Boolean semiring $2 = \{0,1\}$
yields the finite powerset monad $\powersetfin = \Mlt_{2}$.

{\em(4)}~Analogously to the previous example one defines the distribution
monad $\Dstr_{S}$ for a semiring $S$ by:
$$\begin{array}{rcl}
\Dstr_{S}(X)
& = &
\set{\varphi\colon X\rightarrow S}{\support(\varphi)
   \mbox{ is finite and }\sum_{x\in X}\varphi(x) = 1},
\end{array}$$

\noindent Elements of $\Dstr_{S}(X)$ are convex combinations
$s_{1}x_{1}+\cdots+s_{k}x_{k}$ where $\sum_{i}s_{i} = 1$. Unit and
multiplication can be defined as before. This multiplication is
well-defined since:
$$\textstyle
\sum_{x}\mu(\sum_{i}s_{i}\varphi_{i})(x)
=
\sum_{x}\sum_{i}s_{i}\cdot \varphi_{i}(x)
=
\sum_{i}s_{i}\cdot \big(\sum_{x}\varphi_{i}(x)\big)
=
\sum_{i}s_{i}
=
1.$$

\noindent For the semiring $\nonnegR$ of non-negative real numbers one
obtains the familiar distribution monad $\Dstr_{\nonnegR}$ with
elements $\sum_{i}r_{i}x_{i}$ containing probabilities
$r_{i}\in\unitR$ summing up to 1.  Whenever we write $\Dstr$ without
semiring subscript we refer to this $\Dstr_{\nonnegR}$.  For the
two-element semiring $2 = \{0,1\}$---with join $\disjun$ as sum and
meet $\conjun$ as multiplication---the monad $\Dstr_{2}$ is the
non-empty finite powerset monad $\powersetnefin$.
\end{example}

The inclusion maps $\Dstr_{S}(X) \hookrightarrow \Mlt_{S}(X)$
are natural and commute with the units and multiplications of the
two monads, and thus form an example of a ``map of monads''.

\begin{definition}
\label{MonadAlgDef}
Given a monad $T = (T, \eta, \mu)$ as in the previous definition, one
defines an algebra of this monad as a map $\alpha\colon
T(X)\rightarrow X$ satisfying two requirements, expressed
via the diagrams:
$$\xymatrix@C-1pc@R-.5pc{
X\ar[rr]^-{\eta_X}\ar@{=}[drr] & & T(X)\ar[d]^{\alpha}
& &
T^{2}(X)\ar[rr]^-{\mu_X}\ar[d]_{T(\alpha)} & & T(X)\ar[d]^{\alpha} \\
& & X
& &
T(X)\ar[rr]_-{\alpha} & & X
}$$

\noindent We shall write $\Alg(T)$ for the category with such algebras
as objects. A morphism $\smash{\big(T(X)\stackrel{\alpha}{\rightarrow}
  X\big) \stackrel{f}{\longrightarrow}
  \big(T(Y)\stackrel{\beta}{\rightarrow} Y\big)}$ in $\Alg(T)$ is
a map $f\colon X\rightarrow Y$ between the underlying sets satisfying
$f \after \alpha = \beta \after T(f)$.
\end{definition}

There is an obvious forgetful functor $U\colon \Alg(T) \rightarrow
\Sets$ that maps an algebra to its underlying set:
$\smash{U\big(T(X)\stackrel{\alpha}{\rightarrow} X\big) = X}$. It has a left
adjoint mapping a set $Y$ to the multiplication $\mu_Y$, as algebra
$T(T(Y)) \rightarrow T(Y)$ on $T(Y)$.

We shall briefly review what algebras are of the monads (1)--(3) in
Example~\ref{MonadEx}. Elaborating all details requires some amount of
work. The algebras of the fourth (distribution) monad will be
characterised in the next section.

\begin{example}
\label{MonadAlgEx}
{\em(1)}~The category of algebras of the monad $\widehat{M} =
M\times(-)$ for a monoid $M$ is precisely the category of $M$-actions
and their morphisms. Such an action consists of a scalar
multiplication map $\bullet\colon M\times X\rightarrow X$ satisfying
two equations, $0\mathrel{\bullet} x = x$ and $(a+b)\mathrel{\bullet} x
= a\mathrel{\bullet} (b\mathrel{\bullet} x)$, corresponding to the
two diagrams in Definition~\ref{MonadAlgDef}.

{\em(2)}~Algebras $\alpha\colon \powerset(X)\rightarrow X$ for the
powerset monad $\powerset$ correspond to a join operation of a
complete lattice. Such $\alpha$ yields an partial order $x\leq y
\Leftrightarrow \alpha(\{x,y\}) = y$, with $\alpha(U)$ as least
upperbound of the elements in $U$. Algebra homomorphisms correspond to
``linear'' functions that preserve all joins.

{\em(3)}~An algebra $\alpha\colon\Mlt_{S}(X)\rightarrow X$ for the
multiset monad corresponds to a monoid structure on $X$---given by
$x+y = \alpha(1x + 1y)$---together with a scalar multiplication
$\bullet \colon S\times X\rightarrow X$ given by $s\mathrel{\bullet} x
= \alpha(sx)$. It preserves the additive structure (of $S$ and of $X$)
in each coordinate separately. This makes $X$ a semimodule, for the
semiring $S$. Conversely, such an $S$-semimodule structure on a
commutative monoid $M$ yields an algebra $\Mlt_{S}(M)\rightarrow M$ by
$\sum_{i}s_{i}x_{i} \mapsto \sum_{i}s_{i}\mathrel{\bullet}x_{i}$.
Thus the category of algebras $\Alg(\Mlt_{S})$ is equivalent to the
category $\Cat{SMod}_{S}$ of $S$-semimodules.
\end{example}

We continue this section with two basic results, which are stated
without proof, but with a few subsequent pointers.

\begin{theorem}
\label{MonadAlgStructThm}
For a monad $T$ on $\Sets$, the category $\Alg(T)$ of algebras is:
\begin{enumerate}
\item both complete and cocomplete, so has all limits and colimits;

\item symmetric monoidal/tensorial closed in case the monad $T$ is
  ``commutative''. \QED
\end{enumerate}
\end{theorem}

A category of algebras is always ``as complete'' as its underlying
category, see \textit{e.g.}~\cite{Manes74,BarrW85}. Since \Sets is
complete, so is $\Alg(T)$.  Cocompleteness is special for algebras
over \Sets and follows from a result of Linton's, see~\cite[\S~9.3,
  Prop.~4]{BarrW85}.

Monoidal structure in categories of algebras goes back
to~\cite{Kock71a,Kock71b}. Each monad on \Sets is strong, via a
``strength'' map $\st\colon X\times T(Y) \rightarrow T(X\times Y)$
given as $\st(x,v) = T(\lam{y}{\tuple{x, y}})(v)$. There is
also a swapped version $\st'\colon T(X)\times Y \rightarrow T(X\times
Y)$ given by $\st'(u, y) = T(\lam{x}{\tuple{x,y}})(u)$. There are now
in principle two maps $T(X)\times T(Y) \rightrightarrows T(X\times
Y)$, namely:
$$\xymatrix@R-2pc{
& T(T(X)\times Y)\ar[r]^-{T(\st')} &
   T^{2}(X\times Y) \ar[dr]^-{\mu} & \\
T(X)\times T(Y)\ar[ur]^-{\st}\ar[dr]_-{\st'}  & & & T(X\times Y) \\
& T(X\times T(Y))\ar[r]_-{T(\st)} &
   T^{2}(X\times Y) \ar[ur]_-{\mu} & 
}$$

\noindent The monad $T$ is called commutative if these two composites
$T(X)\times T(Y) \rightrightarrows T(X\times Y)$ are the same. 

The monad $\widehat{M} = M\times(-)$ in Example~\ref{MonadEx} is
commutative if and only if $M$ is a commutative monoid. The other
three examples $\powerset, \Mlt_{S}$ and $\Dstr$ are commutative.

We proceed with some elementary observations about the functoriality
in the semiring $S$ of the monad constructions $\Mlt_S$ and $\Dstr_S$
from Example~\ref{MonadEx}.

\begin{lemma}
\label{SemiRingHomLem}
Let $h\colon S\rightarrow S'$ be a homomorphism of semirings (preserving
both $0,+$ and $1,\cdot$). It yields:
\begin{enumerate}
\item homomorphisms of monads $\Mlt_{S} \rightarrow \Mlt_{S'}$ and
  $\Dstr_{S}\rightarrow \Dstr_{S'}$ by post-composi\-tion:
  $\varphi\mapsto f\after \varphi$, or equivalently,
  $\sum_{i}s_{i}x_{i} \mapsto \sum_{i}h(s_{i})x_{i}$;

\item functors, in the opposite direction, between the associated
  categories of algebras $\Alg(\Mlt_{S'})\rightarrow \Alg(\Mlt_{S})$
  and $\Alg(\Dstr_{S'})\rightarrow \Alg(\Dstr_{S})$, via
  pre-composi\-tion with the monad map from the previous point. \QED
\end{enumerate}
\end{lemma}

\begin{definition}
\label{SemiringVersionsDef}
A semiring $S$ is called zerosumfree if $x+y=0$ implies both $x=0$ and
$y=0$.  It is integral if it has no zero divisors: $x\cdot y = 0$
implies either $x=0$ or $y=0$. And it is called a semifield if it is
non-trivial (\textit{i.e.}~$0\neq 1$), zerosumfree, integral and each
non-zero element $s\in S$ has a multiplicative inverse $s^{-1} =
\frac{1}{s}\in S$.
\end{definition}

The semiring $\mathbb{N}$ of natural numbers is non-trivial, zerosumfree
and integral. The semirings $\mathbb{Q}_{\geq 0}$ and
$\nonnegR$ of nonnegative rational and real numbers are
examples of semifields. As is well-known, the quotient ring
$\mathbb{Z}_{n} = \mathbb{Z}/n\mathbb{Z}$ of integers modulo $n$ is
integral if $n$ is prime.

For a non-trivial, zerosumfree, integral
semiring $S$ there is a homomorphism of semirings $h\colon
S\rightarrow 2$ given by $h(x) = 0$ iff $x=0$. For such a semiring
Lemma~\ref{SemiRingHomLem} yields functors:
\begin{equation}
\label{JoinLatHomEqn}
\xymatrix@R-2pc{
\Cat{FJL} = \Alg(\powersetfin) = \Alg(\Mlt_{2})\ar[r] & \Alg(\Mlt_{S}) \\
\Cat{BJL} = \Alg(\powersetnefin) = \Alg(\Dstr_{2})\ar[r] & \Alg(\Dstr_{S})
}
\end{equation}

\noindent where \Cat{FJL} is the category of finite join lattices,
with finite joins $(0,\vee)$, and \Cat{BJL} the category of binary
join lattices, with join $\vee$ only (and thus joins over all
non-empty finite subsets). This construction uses for a lattice $L$
the scalar multiplication: 
$$\xymatrix{
S\times L\ar[r] & L \quad\mbox{given by}\quad (s,x)\ar@{|->}[r] &
   {\left\{\begin{array}{ll} 0 & \mbox{if }s=0 \\ x & \mbox{otherwise.}
   \end{array}\right.}
}$$

\noindent and thus the interpretation $s_{1}x_{1}+\cdots+s_{n}x_{n}
\longmapsto x_{1}\disjun \cdots \disjun x_{n}$, assuming $s_{i}\neq 0$
for each $i$. For the distribution monad $\Dstr$ this involves a
non-empty join, since the $s_i$ must add up to 1. We can do the same
for a meet semilattice $(K,\wedge,1)$ since $K\op$ with order reserved
is a join semilattice. The scalar multiplication becomes $(0,x)\mapsto
1$ and $(s,x)\mapsto x$ if $s\neq 0$, so that the induced semimodule
structure is, assuming $s_{i}\neq 0$,
\begin{equation}
\label{MSLsemimoduleEqn}
\begin{array}{rcl}
s_{1}x_{1}+\cdots+s_{n}x_{n} 
& \longmapsto &
x_{1}\conjun \cdots \conjun x_{n}.
\end{array}
\end{equation}

\auxproof{
Clearly, $h\colon S\rightarrow 2$ given by $h(x) = 0$ iff $x=0$
satisfies $h(0) = 0$ and $h(1) = 1$, and:
$$\begin{array}{rcl}
h(x+y) = 0
& \Leftrightarrow &
x+y = 0 \\
& \Leftrightarrow &
x=0 \mbox{ and } y=0 \\
& \Leftrightarrow &
h(x)=0 \mbox{ and } h(y) = 0 \\
& \Leftrightarrow &
h(x) \disjun h(y) = 0 \\
h(x\cdot y) = 0
& \Leftrightarrow &
x\cdot y = 0 \\
& \Leftrightarrow &
x = 0 \mbox{ or } y = 0 \\
& \Leftrightarrow &
h(x) = 0 \mbox{ or } h(y) = 0 \\
& \Leftrightarrow &
h(x) \conjun h(y) = 0.
\end{array}$$

\noindent Additionally we check that the scalar multiplication
$\scalar \colon S\times L\rightarrow L$ uses the required properties
of $S$. Non-triviality $0\neq 1$ ensures that $1\scalar x = x$.
Clearly we have $s\scalar 0 = 0$.  For the composition requirement
$s\scalar (t\scalar x) = (s\cdot t)\scalar x$ consider the following
three cases.
\begin{itemize}
\item if $t=0$, then $s\scalar (t\scalar x) = s\scalar 0 = 0 = 0\scalar 0
= (s\cdot t)\scalar x$;

\item if $t\neq 0$ but $s=0$ we get $s\scalar (t\scalar x) = 0\scalar
  x = 0 = 0\scalar 0 = (s\cdot t)\scalar x$;

\item if $t\neq 0$ and $t\neq 0$ then $s\cdot t\neq 0$ because $S$ is
  integral, so that $s\scalar (t\scalar x) = s\scalar x = x = (s\cdot
  t)\scalar x$.
\end{itemize}

\noindent Similarly we check $(s+t)\scalar x = s\scalar x \disjun
t\scalar x$.
\begin{itemize}
\item if $t=0$, then $(s+t)\scalar x = s\scalar x = s\scalar x \disjun 0
  = s\scalar x \disjun t\scalar x$

\item if $t\neq 0$ but $s=0$ we get $(s+t)\scalar x = t\scalar x = 0 \disjun
  t\scalar x = s\scalar x \disjun t\scalar x$

\item if $t\neq 0$ and $t\neq 0$ then $s+t\neq 0$ by zerosumfreeness
  and so $(s+t)\scalar x = x = x\disjun x = s\scalar x \disjun
  t\scalar x$.
\end{itemize}

\noindent Finally we check $s\scalar (x\disjun y) = s\scalar x
\disjun s\scalar y$:
\begin{itemize}
\item if $s=0$, then $s\scalar (x\disjun y) = 0 = 0\disjun 0 =
  s\scalar x \disjun s\scalar y$ 

\item if $s\neq 0$, then $s\scalar (x\disjun y) = x\disjun y = s\scalar x
\disjun s\scalar y$.
\end{itemize}
}

The next construction goes back to~\cite{Stone49} and occurs in many
places (see \textit{e.g.}~\cite{PulmannovaG98,Keimel08}) but is
usually not formulated in the following way. It can be understood as a
representation theorem turning a convex set into a semimodule.

\begin{proposition}
\label{DstrMltAlgAdjProp}
Let $S$ be a semifield. The functor $U\colon \Alg(\Mlt_{S})
\rightarrow \Alg(\Dstr_{S})$ induced by the map of monads $\Dstr_{S}
\Rightarrow \Mlt_{S}$ has a left adjoint.
\end{proposition}

\begin{proof}
  Assume an algebra $\alpha\colon\Dstr_{S}(X)\rightarrow X$ and write
  $S_{\neq 0} = \setin{s}{S}{s\neq 0}$ for the set of non-zero
  elements. We shall turn it into a semimodule $F(X)$, where:
$$\begin{array}{rcl}
F(X)
& = &
\{0\} + S_{\neq0}\times X,
\end{array}$$

\noindent with addition for $u,v\in F(X)$,
$$\begin{array}{rcl}
u+v
& = &
\left\{\begin{array}{ll}
0 & \mbox{if $u=0$ and $v=0$} \\
u & \mbox{if $v=0$} \\
v & \mbox{if $u=0$} \\
(s+t, \alpha(\frac{s}{s+t}x + \frac{t}{s+t}y)) 
  & \mbox{if $u=(s,x)$ and $v=(t,y)$.}
\end{array}\right.
\end{array}$$

\noindent It is well-defined by zerosumfreeness of $S$. By
construction, $0\in F(X)$ is the neutral element for this $+$. A
scalar multiplication $\bullet\colon S\times F(X)\rightarrow F(X)$ is
defined as:
$$\begin{array}{rcl}
s\scalar u
& = &
\left\{\begin{array}{ll}
0 & \mbox{if $u=0$ or $s=0$} \\
(s\cdot t, x) & \mbox{if }u = (t,x)\mbox{ and }s\neq 0.
\end{array}\right.
\end{array}$$

\noindent Well-definedness follows because $S$ is integral. Obviously,
$1\scalar u = u$ (because $S$ is non-trivial) and $r\scalar (s\scalar u)
= (r\cdot s)\scalar u$. This makes $F(X)$ a semimodule over $S$.

\auxproof{
We check that $F(X)$ is a semimodule. Commutativity of $+$ is 
obvious, so we look at associativity. Since $u+v = 0$ iff both
$u=0$ and $v=0$, this is obvious for $u=0$ or $v=0$ or $w=0$.
Hence we consider $u=(r,x)$, $v=(s,y)$, $w=(t,z)$. Associativity
then holds because:
$$\begin{array}{rcl}
u + (v+w)
& = &
\alpha(\frac{r}{r+s+t}x + \frac{s+t}{r+s+t}
   \alpha(\frac{s}{s+t}y + \frac{t}{s+t}z)) \\
& = &
\alpha(\frac{r}{r+s+t}\alpha(1x) + \frac{s+t}{r+s+t}
   \alpha(\frac{s}{s+t}y + \frac{t}{s+t}z)) \\
& = &
\alpha((\frac{r}{r+s+t}\cdot 1)x + (\frac{s+t}{r+s+t}\cdot \frac{s}{s+t})y + 
   (\frac{s+t}{r+s+t}\cdot\frac{t}{s+t})z) \\
& = &
\alpha(\frac{r}{r+s+t}x + \frac{s}{r+s+t}y + \frac{t}{r+s+t}z) \\
& = &
\alpha((\frac{r+s}{r+s+t}\cdot\frac{r}{r+s})y + 
   (\frac{r+s}{r+s+t}\cdot\frac{s}{r+s})z + (\frac{t}{r+s+t}\cdot 1)z) \\
& = &
\alpha(\frac{r+s}{r+s+t}\alpha(\frac{r}{r+s}y + \frac{s}{r+s}z) +
   \frac{t}{r+s+t}\alpha(1z)) \\
& = &
\alpha(\frac{r+s}{r+s+t}\alpha(\frac{r}{r+s}y + \frac{s}{r+s}z) +
   \frac{t}{r+s+t}z) \\
& = &
(u+v)+w
\end{array}$$

\noindent The zero-ary distributivity properties $0\scalar u = 0$ and
$s\scalar 0 = 0$ hold by construction. Further, for $s\neq 0$,
$u=(r,x), v=(t,y)$,
$$\begin{array}{rcl}
s\scalar (u+v)
& = &
s\scalar (r+t, \alpha(\frac{r}{r+t}x + \frac{t}{r+t}y)) \\
& = &
(s\cdot (r+t), \alpha(\frac{r}{r+t}x + \frac{t}{r+t}y)) \\
& = &
(s\cdot r+s\cdot t, \alpha(\frac{s\cdot r}{s\cdot r+s\cdot t}x + 
   \frac{s\cdot t}{s\cdot r+s\cdot t}y)) \\
& = &
(s\cdot r, x) + (s\cdot t, y) \\
& = &
s\scalar u + s\scalar v.
\end{array}$$

\noindent Similarly, if also $r\neq 0$,
$$\begin{array}{rcl}
(r+s)\scalar v
& = &
((r+s)\cdot t, y) \\
& = &
(r\cdot t + s\cdot t, y) \\
& = &
(r\cdot t + s\cdot t, \alpha(1y)) \\
& = &
(r\cdot t + s\cdot t, \alpha(\frac{r\cdot t}{r\cdot t + s\cdot t}y +
   \frac{s\cdot t}{r\cdot t + s\cdot t}y) \\
& = &
(r\cdot t, y) + (s\cdot t, y) \\
& = &
r\scalar v + s\scalar v.
\end{array}$$
}

Next we show that $F$ yields a left adjoint to $U\colon \Alg(\Mlt_{S})
\rightarrow \Alg(\Dstr_{S})$, via the following bijective correspondence.
For a semimodule $Y$,
$$\begin{prooftree}
{\xymatrix{ X\ar[r]^-{f} & U(Y)}}
   \rlap{\qquad in $\Alg(\Dstr_{S})$} 
\Justifies
{\xymatrix{F(X)\ar[r]_-{g} & Y}}
   \rlap{\qquad in $\Alg(\Mlt_{S})$}
\end{prooftree}$$

\noindent It works as follows.
\begin{itemize}
\item Given $f\colon X\rightarrow U(Y)$ in $\Alg(\Dstr_{S})$ define
  $\overline{f}\colon F(X)\rightarrow Y$ by $\overline{f}(0) = 0$ and
  $\overline{f}(r,x) = r\scalar f(x)$ where $\scalar$ is scalar
  multiplication in $Y$. This yields a homomorphism of semimodules,
  \textit{i.e.}~a homomorphism of $\Mlt_S$-algebras.

\item Conversely, given $g\colon F(X)\rightarrow Y$ take
  $\overline{g}\colon X\rightarrow U(Y)$ to be $\overline{g}(x) =
  g(1,x)$. This yields a map of $\Dstr_S$-algebras.
\end{itemize}

Finally we check that we actually have a bijective correspondence:
$$\overline{\overline{f}}(x)
=
\overline{f}(1,x)
=
1\scalar f(x)
=
f(x).$$

\noindent Similarly, $\overline{\overline{g}}(0) = 0$ and:
$$\overline{\overline{g}}(r,x)
=
r\scalar \overline{g}(x) 
=
r\scalar g(1,x)
=
g(r\scalar (1,x))
=
g(r,x).\eqno{\QEDbox}$$

\auxproof{
We first check that $\overline{f}\colon F(X)\rightarrow Y$ is a
map of semimodules. It preserves $0\in F(X)$ by construction, and
for $u=(r,x), v=(t,y)\in F(X)$:
$$\begin{array}{rcl}
\overline{f}(u+v)
& = &
\overline{f}(r+t, \alpha(\frac{r}{r+t}x + \frac{t}{r+t}y)) \\
& = &
(r+t)\scalar f(\alpha(\frac{r}{r+t}x + \frac{t}{r+t}y)) \\
& = &
(r+t)\scalar \beta(\frac{r}{r+t}f(x) + \frac{t}{r+t}f(y)) \\
& & \qquad\mbox{where $\beta$ is the $\Mlt_S$-algebra on $Y$} \\
& = &
(r+t)\scalar (\frac{r}{r+t}\scalar f(x) + \frac{t}{r+t}\scalar f(y)) \\
& = &
\big((r+t)\cdot \frac{r}{r+t}\big)\scalar f(x) + 
   \big((r+t)\cdot \frac{t}{r+t}\big)\scalar f(y)) \\
& = &
r\scalar f(x) + t\scalar f(y) \\
& = &
\overline{f}(u) + \overline{f}(v). \\
\overline{f}(s\scalar u)
& = &
\overline{f}(s\cdot r, x) \\
& = &
(s\cdot r)\scalar f(x) \\
& = &
s\scalar (r\scalar f(x)) \\
& = &
s\scalar \overline{f}(u).
\end{array}$$

\noindent Similarly, $\overline{g}$ is a map of algebras: for
a convex sum $\sum_{i}s_{i}x_{i}$ in $\Dstr_S(X)$,
$$\begin{array}{rcl}
(\overline{g} \after \alpha)(\sum_{i}s_{i}x_{i})
& = &
g(1, \alpha(\sum_{i}s_{i}x_{i})) \\
& = &
g(\sum_{i}s_{i}, \alpha(\sum_{i}\frac{s_{i}}{\sum_{i}s_{i}}x_{i})) \\
& = &
g(\sum_{i}(s_{i},x_{i})) \\
& = &
g(\sum_{i}s_{i}\scalar (1,x_{i})) \\
& = &
\sum_{i}s_{i}\scalar g(1, x_{i}) \\
& = &
\sum_{i}s_{i}\scalar \overline{g}(x_{i}) \\
& = &
\beta(\sum_{i}s_{i}\overline{g}(x_{i})) \\
& = &
(\beta \after \Dstr_{S}(\overline{g}))(\sum_{i}s_{i}x_{i}),
\end{array}$$

\noindent where $\beta$ is the $\Mlt_S$-algebra on $Y$.
}
\end{proof}

\section{Convex Sets}\label{ConvexSetSec}

This section introduces convex structures---or simply, convex
sets---as described in~\cite{Gudder79} and recalls that such
structures can also be described as algebras of the distribution monad
$\Dstr$ from Example~\ref{MonadEx}~(4).

\begin{definition}
\label{ConvexSetDef}
A convex set consists of a set $X$ together with a ternary operation
$\tuple{-,-,-} \colon \unitR \times X \times X \rightarrow X$
satisfying the following four requirements, for all $r\in\unitR$ and
$x,y,z\in X$.
\begin{enumerate}
\item $\tuple{r,x,y} = \tuple{1-r,y,x}$

\item $\tuple{r,x,x} = x$

\item $\tuple{0,x,y} = y$

\item $\tuple{r, x, \tuple{s,y,z}} = 
   \tuple{r+(1-r)s, \; \tuple{\frac{r}{(r+(1-r)s)}, x, y}, \; z}$,
assuming that $(r+(1-r)s) \neq 0$.
\end{enumerate}

A morphism of convex structures $(X, \tuple{-,-,-}_{X}) \rightarrow
(Y, \tuple{-,-,-}_{Y})$ consists of an ``affine'' function $f\colon
X\rightarrow Y$ satisfying $f(\tuple{r,x,x'}_{X}) =
\tuple{r,f(x),f(x')}_{Y}$, for all $r\in\unitR$ and $x,x'\in X$. This
yields a category \Cat{Conv}.
\end{definition}

A convex set is sometimes called a barycentric algebra, using
terminology from~\cite{Stone49}. The tuple $\tuple{r,x,y}$ can also be
written as labeled sum $x \mathrel{+_r} y$, like in~\cite{Keimel08},
but the fourth condition becomes a bit difficult to read with this
notation.



The next result recalls an alternative description of convex
structures and their homomorphisms, namely as algebras of a monad. It
goes back to~\cite{Swirszcz74} and also applies to compact Hausdorff
spaces~\cite{Keimel08} or Polish spaces~\cite{Doberkat06}. For
convenience, a proof sketch is included. Recall that the notation
$\Dstr$ without subscript refers to the distribution monad
$\Dstr_{\nonnegR}$ for the semiring $\nonnegR$ of non-negative real
numbers.

\begin{theorem}
\label{ConvexSetAlgThm}
The category $\Cat{Conv}$ of (real) convex structures is isomorphic to
the category $\Alg(\Dstr)$ of Eilenberg-Moore algebras of the
distribution monad. Hence convex sets are algebraic over sets.
\end{theorem}

\begin{proof}
Given an algebra $\alpha\colon \Dstr(X)\rightarrow X$ on a set $X$ one
defines an operation $\tuple{-,-,-} \colon \unitR \times X \times X
\rightarrow X$ by:
\begin{equation}
\label{ConvFromAlgEqn}
\begin{array}{rcl}
\tuple{r,x,y}
& = &
\alpha(rx + (1-r)y).
\end{array}
\end{equation}

\noindent It is not hard to show that the four requirements from
Definition~\ref{ConvexSetDef} hold.

\auxproof{
\begin{enumerate}
\item $\tuple{r,x,y} = \alpha(rx + (1-r)y) = \alpha((1-r)y + rx) =
\tuple{1-r, y, x}$.

\item $\tuple{r,x,x} = \alpha(rx + (1-r)x) = \alpha(1x) = \alpha(\eta(x)) = x$.

\item $\tuple{0,x,y} = \alpha(0x + (1-0)y) = \alpha(1y) = y$.

\item Using that $\alpha \after \mu = \alpha \after \Dstr(\alpha)
  \after \alpha$ we show that both sides are equal.
$$\begin{array}{rcl}
\tuple{r, x, \tuple{s,y,z}}
& = &
\alpha(rx + (1-r)\alpha(sy + (1-s)z)) \\
& = &
\alpha(r\alpha(1x) + (1-r)\alpha(sy + (1-s)z)) \\
& = &
\alpha(rx + (1-r)sy + (1-r)(1-s)z).
\end{array}$$

\noindent Similarly,
$$\begin{array}{rcl}
\lefteqn{\textstyle
   \tuple{r+(1-r)s, \; \tuple{\frac{r}{r+(1-r)s}, x, y}, \; z}} \\
& = &
\tuple{r+(1-r)s, \; \alpha(\frac{r}{r+(1-r)s}x + 
   \big(1-\frac{r}{r+(1-r)s}\big)y), \; z} \\
& = &
\alpha\Big((r+(1-r)s)\alpha(\frac{r}{r+(1-r)s}x + 
   \frac{(1-r)s}{r+(1-r)s}y) \\
& & \qquad +\; (1 - (r+(1-r)s))z\Big) \qquad \mbox{since} \\
& & \qquad 
   1-\frac{r}{r+(1-r)s} = 
   \frac{r+(1-r)s}{r+(1-r)s} - \frac{r}{r+(1-r)s} =
   \frac{(1-r)s}{r+(1-r)s} \\
& = &
\alpha\Big((r+(1-r)s)\alpha(\frac{r}{(r+(1-r)s)}x + 
   \frac{(1-r)s}{(r+(1-r)s)}y) \\
& & \qquad +\; (1 - r)(1-s)\alpha(1z)\Big) \qquad \mbox{since} \\
& & \qquad
   1 - (r+ (1-r)s) = 1 - r - s + rs = (1-r)(1-s) \\
& = &
\alpha(ra + (1-r)sy + (1-r)(1-s)z).
\end{array}$$
\end{enumerate}

It is not hard to see that an algebra map
$(\Dstr(X)\stackrel{\alpha}{\rightarrow} X)
\stackrel{f}{\longrightarrow} (\Dstr(Y)\stackrel{\beta}{\rightarrow}
Y)$ yields a map of convex structures, since:
$$\begin{array}{rcl}
f(\tuple{r,x,x'}_{X})
\hspace*{\arraycolsep}=\hspace*{\arraycolsep}
f(\alpha(rx + (1-r)x'))
& = &
\beta(\Dstr(f)(rx + (1-r)x')) \\
& = &
\beta(rf(x) + (1-r)f(x')) \\
& = &
\tuple{r, f(x), f(x')}_{Y}.
\end{array}$$
}

Conversely, given a convex set $X$ with operation $\tuple{-,-,-}$ one
defines a function $\alpha\colon \Dstr(X)\rightarrow X$ inductively by:
\begin{equation}
\label{AlgFromConvRecEqn}
\hspace*{-2em}\begin{array}{rcl}
\lefteqn{\alpha(r_{1}x_{1} + \cdots + r_{n}x_{n})} \\
& = &
\left\{\begin{array}{ll}
x_{1} & \mbox{if $r_{1}=1$, so $r_{2}=\cdots= r_{n}$ \rlap{$=0$}} \\
\tuple{r_{1}, x_{1}, \alpha(\frac{r_2}{1-r_1}x_{2} + \cdots +
    \frac{r_n}{1-r_1}x_{n})} & \mbox{otherwise, \textit{i.e.}~$r_{1}<1$.}
\end{array}\right.
\end{array}
\end{equation}

\noindent Repeated application of this definition yields:
\begin{equation}
\label{AlgFromConvEqn}
\hspace*{-.8em}\begin{array}{rcl}
\lefteqn{\alpha(r_{1}x_{1} + \cdots + r_{n}x_{n})} \\
& = &
\tuple{r_{1}, x_{1}, \tuple{\frac{r_{2}}{1-r_{1}}, x_{2},
   \tuple{\frac{r_{3}}{1-r_{1}-r_{2}}, x_{3}, \tuple{ \ldots,
   \tuple{\frac{r_{n-1}}{1-r_{1}-\cdots-r_{n-2}}, x_{n-1}, x_{n}}\ldots}}}}.
\end{array}
\end{equation}

\noindent One first has to show that the function $\alpha$
in~\eqref{AlgFromConvRecEqn} is well-defined, in the sense that it
does not depend on permutations of summands, see
also~\cite[Lemma~2]{Stone49}. Via some elementary calculations one
checks that exchanging the summands $r_{i}x_{i}$ and $r_{i+1}x_{i+1}$
produces the same result. In a next step one proves the algebra
equations: $\alpha \after \eta = \idmap$ and $\alpha \after \mu =
\alpha \after \Dstr(\alpha)$.  The first one is easy, since
$\alpha(\eta(a)) = \alpha(1a) = a$, directly by
applying~\eqref{AlgFromConvRecEqn}. The second one requires more
work. Explicitly, it amounts to:
\begin{equation}
\label{AlgFromConvMuEqn}
\begin{array}{rcl}
\alpha\big(\sum_{i\leq n}r_{i}\alpha(\sum_{j\leq m_i}s_{ij}x_{ij})\big)
& = &
\alpha\big(\sum_{i\leq n}\sum_{j\leq m_i} (r_{i}s_{ij})x_{ij}\big).
\end{array}
\end{equation}

\noindent For the proof the following auxiliary result is
convenient. It handles nested tuples in the second argument of a
triple $\tuple{-,-,-}$, just like condition~(4) in
Definition~\ref{ConvexSetDef} deals with nested structure in the third
argument. In a general convex structure one has:
\begin{equation}
\label{AlgFromConvMuAuxLem}
\begin{array}{rcl}
\tuple{r, \tuple{s, x, y}, z}
& = &
\tuple{rs, x, \tuple{\frac{r(1-s)}{1-rs}, y, z}}.
\end{array}
\end{equation}

\noindent assuming $rs\neq 1$. The rest is then left to the reader. \QED

\auxproof{
We first check this second formulation~\eqref{AlgFromConvEqn}. 
Before $x_3$ one has:
$$\frac{\frac{r_{3}}{1-r_{1}}}{1-\frac{r_{2}}{1-r_{1}}}
=
\frac{\frac{r_{3}}{1-r_{1}}}{\frac{1-r_{1}}{1-r_{1}}-\frac{r_{2}}{1-r_{1}}}
=
\frac{\frac{r_{3}}{1-r_{1}}}{\frac{1-r_{1}-r_{2}}{1-r_{1}}}
=
\frac{r_{3}}{1-r_{1}-r_{2}}$$

\noindent etc., until before $x_n$ one gets:
$$\frac{\frac{r_{n}}{1-r_{1}-\cdots-r_{n-2}}}
   {1-\frac{r_{n-1}}{1-r_{1}-\cdots-r_{n-2}}}
=
\frac{r_{n}}{1-r_{1}-\cdots-r_{n-2}-r_{n-1}}
=
\frac{r_n}{r_{n}}
=
1.
$$


Each permutation is a composition of adjacent transpositions of
summands (swap elements at positions $i$ and $i+1$).  Hence we
concentrate on such swaps, say of summands $r_{i}x_{i}$ and
$r_{i+1}x_{i+1}$. The result of applying $\alpha$ is then calculated
as:
$$\begin{array}{rcl}
\lefteqn{\alpha(r_{1}x_{1} + \cdots r_{i-1}x_{i-1} + r_{i+1}x_{i+1} +
   r_{i}x_{i} + r_{i+2}x_{i+2} + \cdots + r_{n}x_{n})} \\
& = &
\langle r_{1}, x_{1}, \langle \frac{r_{2}}{1-r_{1}}, x_{2}, \ldots 
   \langle \frac{r_{i+1}}{1-r_{1}-\cdots-r_{i-1}}, x_{i+1}, \\
& & \quad \tuple{\frac{r_{i}}{1-r_{1}-\cdots-r_{i-1}-r_{i+1}}, x_{i}, \ldots
   \tuple{\frac{r_{n-1}}{1-r_{1}-\cdots-r_{n-2}}, x_{n-1}, x_{n}}\ldots}
   \rangle\rangle\rangle.
\end{array}$$

\noindent In order to show that these are equal it suffices to show
the equality of:
$$\begin{array}{rcl}
\tuple{s_{i}, x_{i}, \tuple{s_{i+1}, x_{i+1}, b}}
& = &
\tuple{t_{i+1}, x_{i+1}, \tuple{t_{i}, x_{i}, b}},
\end{array}\eqno{(*)}$$

\noindent where:
$$\begin{array}{rclcrcl}
s_{i}
& = &
\frac{r_{i}}{1-r_{1}-\cdots-r_{i-1}}
& \quad &
s_{i+1}
& = &
\frac{r_{i+1}}{1-r_{1}-\cdots-r_{i-1}-r_{i}} \\
t_{i}
& = &
\frac{r_{i}}{1-r_{1}-\cdots-r_{i-1}-r_{i+1}}
& &
t_{i+1}
& = &
\frac{r_{i+1}}{1-r_{1}-\cdots-r_{i-1}}.
\end{array}$$

\noindent We obtain $(*)$ via the equations in
Definition~\ref{ConvexSetDef}, first applied to the left-hand-side:
$$\begin{array}{rcl}
\tuple{s_{i}, x_{i}, \tuple{s_{i+1}, x_{i+1}, b}} 
& = &
\tuple{s_{i} + (1-s_{i})s_{i+1}, \tuple{
   \frac{s_{i}}{s_{i} + (1-s_{i})s_{i+1}}, x_{i}, x_{i+1}}, b} \\
& = &
\tuple{s_{i} + (1-s_{i})s_{i+1}, \tuple{
   1-\frac{s_{i}}{s_{i} + (1-s_{i})s_{i+1}}, x_{i+1}, x_{i}}, b} \\
& = &
\tuple{s_{i} + (1-s_{i})s_{i+1}, \tuple{
   \frac{(1-s_{i})s_{i+1}}{s_{i} + (1-s_{i})s_{i+1}}, x_{i+1}, x_{i}}, b} \\
& = &
\tuple{\frac{r_{i}+r_{i+1}}{1-r_{1}-\cdots-r_{i-1}}, \tuple{
   \frac{r_{i+1}}{r_{i}+r_{i+1}}, x_{i+1}, x_{i}}, b},
\end{array}$$

\noindent since:
$$\begin{array}{rcl}
(1-s_{i})s_{i+1}
& = &
(1-\frac{r_{i}}{1-r_{1}-\cdots-r_{i-1}})
   \frac{r_{i+1}}{1-r_{1}-\cdots-r_{i-1}-r_{i}} \\
& = &
\frac{1-r_{1}-\cdots-r_{i-1}-r_{i}}{1-r_{1}-\cdots-r_{i-1}}
   \frac{r_{i+1}}{1-r_{1}-\cdots-r_{i-1}-r_{i}} \\
& = &
\frac{r_{i+1}}{1-r_{1}-\cdots-r_{i-1}}.
\end{array}$$

\noindent Hence
$$\begin{array}{rclcrcl}
s_{i} + (1-s_{i})s_{i+1}
& = &
\frac{r_{i}+r_{i+1}}{1-r_{1}-\cdots-r_{i-1}}
& \mbox{ and } &
\frac{(1-s_{i})s_{i+1}}{s_{i} + (1-s_{i})s_{i+1}}
& = &
\frac{r_{i+1}}{r_{i}+r_{i+1}}.
\end{array}$$

\noindent Similarly, the right-hand-side of $(*)$ satisfies:
$$\begin{array}{rcl}
\tuple{t_{i+1}, x_{i+1}, \tuple{t_{i}, x_{i}, b}}
& = &
\tuple{t_{i+1} + (1-t_{i+1})t_{i}, \tuple{
   \frac{t_{i+1}}{t_{i+1} + (1-t_{i+1})t_{i}}, x_{i+1}, x_{i}}, b}
\end{array}$$

\noindent which is the same as above, since:
$$\begin{array}{rcl}
(1-t_{i+1})t_{i}
& = &
   (1-\frac{r_{i+1}}{1-r_{1}-\cdots-r_{i-1}})
   \frac{r_{i}}{1-r_{1}-\cdots-r_{i-1}-r_{i+1}} \\
& = &
   \frac{1-r_{1}-\cdots-r_{i-1}-r_{i+1}}{1-r_{1}-\cdots-r_{i-1}}
   \frac{r_{i}}{1-r_{1}-\cdots-r_{i-1}-r_{i+1}} \\
& = &
   \frac{r_{i}}{1-r_{1}-\cdots-r_{i-1}}
\end{array}$$

\noindent and so we get the same outcomes:
$$\begin{array}{rclcrcl}
t_{i+1} + (1-t_{i+1})t_{i}
& = &
\frac{r_{i} + r_{i+1}}{1-r_{1}-\cdots-r_{i-1}}
& \mbox{ and } &
\frac{t_{i+1}}{t_{i+1} + (1-t_{i+1})t_{i}}
& = &
\frac{r_{i+1}}{r_{i}+r_{i+1}}.
\end{array}$$

Towards the $\mu$-equation for $\alpha$ we first prove the
equation~\eqref{AlgFromConvMuAuxLem}: we simply swap arguments and
use~(4) from Definition~\ref{ConvexSetDef}:
$$\begin{array}{rcl}
\tuple{r, \tuple{s, x, y}, z}
& = &
\tuple{1-r, z, \tuple{1-s, y, x}} \\
& = &
\tuple{1-rs, \tuple{\frac{1-r}{1-rs}, z, y}, x}, \quad \mbox{since} \\
& & \quad 
   (1-r) + (1 - (1-r))(1-s) = 1-r + r(1-s) = 1-rs \\
& = &
\tuple{rs, x, \tuple{1-\frac{1-r}{1-rs}, y, z}} \\
& = &
\tuple{rs, x, \tuple{\frac{r(1-s)}{1-rs}, y, z}},
\end{array}$$

\noindent since $1-rs - (1-r) = -rs + r = r(1-s)$.

Now we are in a position to prove the
equation~\eqref{AlgFromConvMuEqn}, by applying this lemma suitably
many times, where we simply assume that all relevant denominators are
non-zero. We start with the right-hand-side of~\eqref{AlgFromConvMuEqn},
using~\eqref{AlgFromConvEqn}.
$$\begin{array}{rcl}
\lefteqn{\alpha\big((r_{1}s_{11})x_{11} + \cdots (r_{1}s_{1m_1})x_{1m_1} + \cdots +
   (r_{n}s_{n1})x_{n1} + \cdots (r_{n}s_{nm_n})x_{nm_n}\big)} \\
& = &
\langle r_{1}s_{11}, x_{11}, \langle \frac{r_{1}s_{12}}{1-r_{1}s_{11}}, x_{12},
  \langle \frac{r_{1}s_{13}}{1-r_{1}s_{11}-r_{1}s_{12}}, x_{13},
  \ldots, \langle \frac{r_{1}s_{1m_{1}}}{1-r_{1}s_{1m_1}}, x_{1m_{1}}, \\
& & \quad
   \langle\frac{r_{2}s_{21}}{1-r_{1}}, x_{21}, 
      \langle\frac{r_{2}s_{22}}{1-r_{1}-r_{2}s_{21}}, x_{22}, \langle \cdots \\
& & \quad
   \langle s_{n1}, x_{n1}, \langle \frac{s_{n2}}{1-s_{n1}}, x_{n2}\cdots,
   \langle\frac{s_{n(m_{n}-1)}}{1-s_{n1}-\cdots-s_{n(m_{n}-2)}}, x_{n(m_{n}-1)},
    x_{nm_{n}}\rangle \cdots \rangle\rangle \cdots \rangle\rangle\rangle\rangle
   \cdots \rangle\rangle\rangle \\
& & \quad\mbox{since }
  \frac{r_{n}s_{n1}}{1-r_{1}-\cdots-r_{n-1}} = \frac{r_{n}s_{n1}}{r_{n}} = s_{n1}
  \mbox{ and }
  \frac{r_{n}s_{n2}}{r_{n}-r_{n}s_{n1}} = \frac{s_{n2}}{1-s_{n1}}
\end{array}$$

\noindent By starting at the other end we get the same:
$$\begin{array}{rcl}
\lefteqn{\textstyle
   \alpha\big(\sum_{i\leq n}r_{i}\alpha(\sum_{j\leq m_i}s_{ij}x_{ij})\big)} \\
& = &
\alpha\big(r_{1}\alpha(s_{11}x_{11}+\cdots+x_{1m_1}) + r_{2}\alpha(\cdots) +
   \cdots + r_{n}\alpha(\cdots)\big) \\
& = &
\tuple{r_{1}, \alpha(s_{11}x_{11}+\cdots+x_{1m_1}), 
   \alpha(\frac{r_2}{1-r_1}\alpha(\cdots) + \cdots +
   \frac{r_n}{1-r_1}\alpha(\cdots))} \\
& = &
\tuple{r_{1}, \tuple{s_{11}, x_{11}, \alpha(\frac{s_{12}}{1-s_{11}}x_{12}
   + \cdots + \frac{s_{1m_1}}{1-s_{11}}x_{1m_1})}, 
   z_{2}}, \\
& & \quad \mbox{for } z_{2} = \alpha(\frac{r_2}{1-r_1}\alpha(\cdots) + \cdots +
   \frac{r_n}{1-r_1}\alpha(\cdots)) \\
& = &
\tuple{r_{1}s_{11}, x_{11}, \tuple{\frac{r_{1}(1-s_{11})}{1-r_{1}s_{11}},
   \alpha(\frac{s_{12}}{1-s_{11}}x_{12}
   + \cdots + \frac{s_{1m_1}}{1-s_{11}}x_{1m_1}), z_{2}}} 
   \quad\mbox{by~\eqref{AlgFromConvMuAuxLem}} \\
& = &
\qquad \vdots \\
& = &
\tuple{r_{1}s_{11}, x_{11}, \tuple{\frac{r_{1}s_{12}}{1-r_{1}s_{11}}, x_{12},
   \tuple{\frac{r_{1}(1-s_{11}-s_{12})}{1-r_{1}s_{11}-r_{1}s_{12}}, 
   \alpha(\frac{s_{13}}{1-s_{11}-s_{12}}x_{13}
   + \cdots + \frac{s_{1m_1}}{1-s_{11}-s_{12}}x_{1m_1}), z_{2}}}} \\
& & \qquad\mbox{since }
   \frac{r_{1}(1-s_{11})}{1-r_{1}s_{11}} \cdot \frac{s_{12}}{1-s_{11}} =
   \frac{r_{1}s_{12}}{1-r_{1}s_{11}} \\
& & \qquad
   \mbox{and } 
   \frac{\frac{r_{1}(1-s_{11})}{1-r_{1}s_{11}} (1 - \frac{s_{12}}{1-s_{11}})}
        {1- \frac{r_{1}s_{12}}{1-r_{1}s_{11}}} = 
      \frac{\frac{r_{1}(1-s_{11})}{1-r_{1}s_{11}} \cdot
            \frac{1-s_{11} - s_{12}}{1-s_{11}}}
        {\frac{1-r_{1}s_{11} - r_{1}s_{12}}{1-r_{1}s_{11}}} = 
   \frac{r_{1}(1-s_{11}-s_{12})}{1-r_{1}s_{11}-r_{1}s_{12}} \\
& & \qquad\mbox{and } 
   \frac{\frac{s_{1i}}{1-s_{11}}}{1 - \frac{s_{12}}{1-s_{11}}} = 
   \frac{s_{1i}}{1-s_{11}-s_{12}} \\
& = &
\qquad \vdots \\
& = &
\langle r_{1}s_{11}, x_{11}, \langle \frac{r_{1}s_{12}}{1-r_{1}s_{11}}, x_{12},
   \langle \frac{r_{1}s_{13}}{1-r_{1}s_{11}-r_{1}s_{12}}, x_{13}, \cdots, \\
& & \qquad 
   \langle \frac{r_{1}s_{1(m_{1}-1)}}{1-r_{1}s_{11}-\cdots-r_{1}s_{1(m_{1}-2)}}, 
    x_{1(m_{1}-1)}, \langle\frac{r_{1}s_{1m_{1}}}{1-r_{1}s_{1m_{1}}}, 
    x_{1m_{1}}, z_{2}\rangle\rangle
   \cdots\rangle\rangle\rangle \\
& & \qquad\mbox{since }
   \frac{r_{1}(1-s_{11}-\cdots s_{1(m_{1}-1)})}
    {1-r_{1}s_{11}-\cdots-r_{1}s_{1(m_{1}-1)}} = 
   \frac{r_{1}s_{1m_{1}}}{1-r_{1}s_{1m_{1}}}.
\end{array}$$

\noindent Continuing with $z_{2}$ we get:
$$\begin{array}{rcl}
z_{2}
& = &
\alpha(\frac{r_2}{1-r_1}\alpha(\cdots) + \cdots +
   \frac{r_n}{1-r_1}\alpha(\cdots)) \\
& = &
\tuple{\frac{r_2}{1-r_{1}}, \alpha(s_{21}x_{21} + 
   \cdots + s_{2n_{2}}x_{2n_{2}}), z_{3}}\rangle\rangle \\
& & \qquad \mbox{where } z_{3} = 
   \alpha(\frac{r_3}{1-r_{1}-r_{2}}\alpha(\cdots) + \cdots +
   \frac{r_n}{1-r_{1}-r_{2}}\alpha(\cdots))  \\
& & \qquad \mbox{since }
   \frac{\frac{r_3}{1-r_{1}}}{1-\frac{r_2}{1-r_{1}}} = 
   \frac{r_3}{1-r_{1}-r_{2}} \\
& = &
\tuple{\frac{r_{2}s_{21}}{1-r_{1}}, x_{21}, 
   \tuple{\frac{r_{2}(1-s_{21})}{1-r_{1}-r_{2}s_{21}},
   \alpha(\frac{s_{22}}{1-s_{21}}x_{22} +  \cdots + 
          \frac{s_{2n_{2}}}{1-s_{21}}x_{2n_{2}}), z_{3}}} \\
& & \qquad\mbox{since }
   \frac{\frac{r_2}{1-r_{1}}(1-s_{21})}{1-\frac{r_2}{1-r_{1}}\cdot s_{21}} =
   \frac{r_{2}(1-s_{21})}{1-r_{1}-r_{2}s_{21}} \\
& = &
\langle \frac{r_{2}s_{21}}{1-r_{1}}, x_{21}, \langle 
   \frac{r_{2}s_{22}}{1-r_{1}-r_{2}s_{21}}, x_{22}, \langle
   \frac{r_{2}(1-s_{21}-s_{22})}{1-r_{1}-r_{2}s_{21}-r_{2}s_{22}}, \\
& & \quad \alpha(\frac{s_{23}}{1-s_{21}-s_{22}}x_{23} +  \cdots + 
          \frac{s_{2n_{2}}}{1-s_{21}-s_{22}}x_{2n_{2}}), 
   z_{3}\rangle\rangle\rangle \\
& & \qquad \mbox{since }
   \frac{r_{2}(1-s_{21})}{1-r_{1}-r_{2}s_{21}}\cdot \frac{s_{22}}{1-s_{21}} =
   \frac{r_{2}s_{22}}{1-r_{1}-r_{2}s_{21}}, \\
& & \qquad\mbox{and }
   \frac{\frac{r_{2}(1-s_{21})}{1-r_{1}-r_{2}s_{21}}(1-\frac{s_{22}}{1-s_{21}})}
        {1-\frac{r_{2}s_{22}}{1-r_{1}-r_{2}s_{21}}} =
   \frac{\frac{r_{2}(1-s_{21})}{1-r_{1}-r_{2}s_{21}} \cdot 
         \frac{1-s_{21}-s_{22}}{1-s_{21}}}
        {\frac{1-r_{1}-r_{2}s_{21}-r_{2}s_{22}}{1-r_{1}-r_{2}s_{21}}} =
   \frac{r_{2}(1-s_{21}-s_{22})}{1-r_{1}-r_{2}s_{21}-r_{2}s_{22}} \\
& & \qquad\mbox{and }
   \frac{\frac{s_{23}}{1-s_{22}}}{1-\frac{s_{22}}{1-s_{21}}} =
   \frac{s_{23}}{1-s_{21}-s_{22}} \\
& = &
\qquad \vdots \\
& = &
\langle \frac{r_{2}s_{21}}{1-r_{1}}, x_{21}, \langle 
   \frac{r_{2}s_{22}}{1-r_{1}-r_{2}s_{21}}, x_{22}, \langle
   \frac{r_{2}s_{23}}{1-r_{1}-r_{2}s_{21}-r_{2}s_{22}}, x_{23}, \cdots \\
& & \quad \langle \frac{r_{2}s_{2(m_{2}-1)}}
   {1-r_{1}-r_{2}s_{21}-\cdots-r_{2}s_{2(m_{2}-2)}}, x_{2(m_{2}-1)}, 
   \langle \frac{r_{2}(1-s_{2m_2})}{1-r_{1}-r_{2}s_{2m_2}}, 
   x_{2m_{2}}, z_{3}\rangle\rangle\cdots\rangle\rangle\rangle \\
& & \qquad\mbox{since }
   \frac{r_{2}(1-s_{21}-\cdots-s_{2(m_{2}-1)})}
        {1-r_{1}-r_{2}s_{21}-\cdots-r_{2}s_{2(m_{2}-1)}} = 
   \frac{r_{2}(1-s_{2m_2})}{1-r_{1}-r_{2}s_{2m_2}}
\end{array}$$

\noindent Finally for $z_n$ we get:
$$\begin{array}{rcl}
z_{n}
& = &
\alpha(\frac{r_{n}}{1-r_{1}-\cdots-r_{n-1}}\alpha(s_{n1}x_{n1} + 
   \cdots + s_{nm_n}x_{nm_{n}})) \\
& = &
\alpha(s_{n1}x_{n1} + \cdots + s_{nm_n}x_{nm_{n}}) \\
& = &
\tuple{s_{n1}, x_{n1}, \tuple{\frac{s_{n2}}{1-s_{n1}}, x_{n2}, \ldots,
   \tuple{\frac{s_{n(m_{n}-1)}}{1-s_{n1}-\cdots-s_{n(m_{n}-2)}}, x_{n(m_{n}-1)},
   x_{n}}}}.
\end{array}$$

\noindent Hence we are done.

Next we need to check that a map $f\colon (X,\tuple{-,-,-}_{X})
\rightarrow (y, \tuple{-,-,-}_{Y})$ of convex structures yields an
algebra homomorphism, between the induced algebras, from
$\alpha\colon\Dstr(X)\rightarrow X)$ to $(\beta\colon \Dstr(Y)
\rightarrow Y)$. This follows easily from~\eqref{AlgFromConvEqn}:
$$\begin{array}{rcl}
\lefteqn{\big(f\after\alpha\big)(r_{1}x_{1} + \cdots + r_{n}x_{n}))} \\
& = &
f\big(\alpha(r_{1}x_{1} + \cdots + r_{n}x_{n})\big) \\
& = &
f\big(\tuple{r_{1}, x_{1}, \tuple{\frac{r_{2}}{1-r_{1}}, x_{2},
    \tuple{ \ldots,\tuple{\frac{r_{n-1}}{1-r_{1}-\cdots-r_{n-2}}, 
   x_{n-1}, x_{n}}_{X}\ldots}_{X}}_{X}}_{X}\big) \\
 & = &
\tuple{r_{1}, f(x_{1}), \tuple{\frac{r_{2}}{1-r_{1}}, f(x_{2}),
   \tuple{\ldots, \tuple{\frac{r_{n-1}}{1-r_{1}-\cdots-r_{n-2}}, 
   f(x_{n-1}), f(x_{n})}_{Y}\ldots}_{Y}}_{Y}}_{Y} \\
& = &
\beta\big(r_{1}f(x_{1}) + \cdots + r_{n}f(x_{n})\big) \\
& = &
\big(\beta \after \Dstr(f)\big)(r_{1}x_{1} + \cdots + r_{n}x_{n}).
\end{array}$$

Finally we need to check that the transformations of an algebra into
a convex structure and vice-versa are each others inverses. First,
starting from an algebra $\alpha\colon\Dstr(X)\rightarrow X$ the
induced convex structure as in~\eqref{ConvFromAlgEqn} yields an
algebra as in~\eqref{AlgFromConvEqn} given by:
$$\begin{array}{rcl}
\lefteqn{r_{1}x_{1} + \cdots + r_{n}x_{n}} \\
& \longmapsto &
\tuple{r_{1}, x_{1}, \tuple{\frac{r_{2}}{1-r_{1}}, x_{2},
    \tuple{ \ldots,\tuple{\frac{r_{n-1}}{1-r_{1}-\cdots-r_{n-2}}, 
   x_{n-1}, x_{n}}\ldots}}} \\
& = &
\alpha(r_{1}x_{1} + (1-r_{1})\tuple{\frac{r_{2}}{1-r_{1}}, x_{2},
    \tuple{ \ldots,\tuple{\frac{r_{n-1}}{1-r_{1}-\cdots-r_{n-2}}, 
   x_{n-1}, x_{n}}\ldots}}) \\
& = &
\alpha(r_{1}x_{1} + (1-r_{1})\alpha(\frac{r_{2}}{1-r_{1}}x_{2} +
    \frac{1-r_{1}-r_{2}}{1-r_{1}}\tuple{\cdots
   \tuple{\frac{r_{n-1}}{1-r_{1}-\cdots-r_{n-2}}, 
   x_{n-1}, x_{n}}\ldots})) \\
& = &
\qquad\vdots \\
& = &
\alpha(r_{1}x_{1} + (1-r_{1})\alpha(\frac{r_{2}}{1-r_{1}}x_{2} +
    \frac{1-r_{1}-r_{2}}{1-r_{1}}\alpha(\cdots \\
& & \qquad
   \frac{1-r_{1}-\cdots-r_{n-1}}{1-r_{1}-\cdots-r_{n-2}}
   \alpha(\frac{r_{n-1}}{1-r_{1}-\cdots-r_{n-2}}x_{n-1}, 
    \frac{r_{n}}{1-r_{1}-\cdots-r_{n-2}}x_{n})\ldots))) \\
& = &
\alpha(r_{1}\alpha(1x_{1}) + (1-r_{1})\alpha(\frac{r_{2}}{1-r_{1}}x_{2} +
    \frac{1-r_{1}-r_{2}}{1-r_{1}}\alpha(\cdots \\
& & \qquad
   \frac{1-r_{1}-\cdots-r_{n-1}}{1-r_{1}-\cdots-r_{n-2}}
   \alpha(\frac{r_{n-1}}{1-r_{1}-\cdots-r_{n-2}}x_{n-1}, 
    \frac{r_{n}}{1-r_{1}-\cdots-r_{n-2}}x_{n})\ldots))) \\
& = &
\alpha(r_{1}x_{1} + r_{2}x_{2} +
    \frac{1-r_{1}-r_{2}}{1-r_{1}}\alpha(\cdots \\
& & \qquad
   \frac{1-r_{1}-\cdots-r_{n-1}}{1-r_{1}-\cdots-r_{n-2}}
   \alpha(\frac{r_{n-1}}{1-r_{1}-\cdots-r_{n-2}}x_{n-1}, 
    \frac{r_{n}}{1-r_{1}-\cdots-r_{n-2}}x_{n})\ldots)) \\
& = &
\qquad\vdots \\
& = &
\alpha(r_{1}x_{1} + \cdots + r_{n}x_{n}).
\end{array}$$

Conversely, starting from a convex structure $(X, \tuple{-,-,-})$
one obtains an algebra $\alpha$ as in~\eqref{AlgFromConvEqn}, which
in turn induces a structure like in~\eqref{ConvFromAlgEqn}, given by:
$$\begin{array}{rcccl}
(r, x, y) 
& \longmapsto &
\alpha(rx + (1-r)y) 
& = &
\tuple{r, x, y}.
\end{array}$$
}
\end{proof}

This theorem now allows us to apply Theorem~\ref{MonadAlgStructThm}
to the category $\Cat{Conv}$ of (real) convex structures. First we may
conclude that it is both complete and cocomplete; also, that the
forgetful functor $\Cat{Conv} \rightarrow \Sets$ has a left adjoint,
giving free convex structures of the form $\Dstr(X)$. And since
$\Dstr$ is a commutative monad, the category $\Cat{Conv}$ is symmetric
monoidal closed: maps $X\otimes Y\rightarrow Z$ in $\Cat{Conv}$
correspond to functions $X\times Y\rightarrow Z$ that are
``bi-homomorphisms'', \textit{i.e.}~homomorphisms of convex structures
in each variable separately. Closedness means that the functors
$(-)\otimes Y$ have a right adjoint, given by
$Y\multimap(-)$. Moreover, $\Dstr(A\times B) \cong \Dstr(A)\otimes
\Dstr(B)$, for set $A,B$.

Theorem~\ref{ConvexSetAlgThm} only applies to the particular monad
$\Dstr = \Dstr_{\nonnegR}$ from our family of monad $\Dstr_S$, for the
special case where the semiring $S$ is given by the non-negative real
numbers $\nonnegR$. Of course, one may try to formulate a notion of
``convex set'', like in Definition~\ref{ConvexSetDef} but more
generally, with respect to a semiring $S$, possibly with some
additional properties. But there is really no need to do so if we are
willing to work in terms of algebras of the monad $\Dstr_S$. In light
of Theorem~\ref{ConvexSetAlgThm} one may consider such algebras as a
generalised form of ``$S$-convex set'', and write $\Cat{Conv}_{S} =
\Alg(\Dstr_{S})$.  The only equations we thus have for such convex
sets are the algebra equations, see Definition~\ref{MonadAlgDef}, with
multiplication equation written explicitly
in~\eqref{AlgFromConvMuEqn}.  Proposition~\ref{DstrMltAlgAdjProp} then
describes an adjunction between $S$-convex sets and $S$-modules. This
line of thinking will be pursued in the next section.

\section{Prime filters in convex sets}\label{FilterSec}

The following definition generalises some familiar notions to
$S$-convex sets, \textit{i.e.}~to
$\Dstr_S$-algebras. In~\cite{Flood81} ideals instead of filters are
used.

\begin{definition}
\label{ConvSubAlgFilterDef}
Let $S$ be a semiring and $\alpha\colon \Dstr_{S}(X)\rightarrow X$ be
an algebra of the monad $\Dstr_S$, making $X$ convex. We write
$(\sum_{i\leq n}s_{i}x_{i})\in\Dstr_{S}(X)$ for an arbitrary convex
combination. A subset $U\subseteq X$ is called a:
\begin{itemize}
\item subalgebra if $\all{i\leq n}{x_{i}\in U}$ implies
  $\alpha(\sum_{i}s_{i}x_{i})\in U$;

\item filter if $\alpha(\sum_{i}s_{i}x_{i})\in U$ implies
$x_{i}\in U$, for each $i$ with $s_{i}\neq 0$;

\item prime filter if it is both a subalgebra and a filter.
\end{itemize}

An element $x\in X$ is called extreme, or a boundary point, if
$\{x\}$ is a prime filter. Often one writes $\partial X$ for
the set of extreme points.
\end{definition}

It is not hard to see that subalgebras are closed under arbitrary
intersections and under directed joins. Hence one can form the least
subalgebra $\overline{V} \subseteq X$ containing an arbitrary set
$V\subseteq X$, by intersection. Explicitly,
$$\begin{array}{rcl}
\overline{V}
& = &
\set{\alpha(\sum_{i}s_{i}x_{i})}{\all{i}{x_{i}\in V}}.
\end{array}$$

\noindent Filters are closed under arbitrary intersections and joins,
hence also prime filters are closed under arbitrary intersections and
directed joins. We shall write $\PrimeFilter(X)$ for the set of prime
filters in a convex set $X$, ordered by inclusion.

Notice that $\partial\unitR = \{0,1\}$ and $\overline{\{0,1\}} =
\unitR$. Hence the unit interval is generated by its boundary points.
In a free convex set $\Dstr_{S}(A)$ the elements $\eta(a) = 1a \in
\Dstr_{S}(A)$, for $a\in A$, are the only boundary points. They also
generate the whole convex set $\Dstr_{S}(A)$. In a quantum context a
state is called pure if it is a boundary point in the convex set of
states, see Section~\ref{EffAlgConvexSec}. The set of mixed states is
the closure of the set of pure states, given by convex combinations of
these pure states.

\begin{lemma}
\label{ConvIdealHomLem}
Assume $S$ is a non-trivial, zerosumfree and integral semiring and $X$ is
an $S$-convex set.  A subset $U\subseteq X$ is a prime filter if and
only if it is the ``true kernel'' $f^{-1}(1)$ of a homomorphism of
convex sets $f\colon X\rightarrow \{0,1\}$. It yields an
order isomomorphism:
$$\begin{array}{rcl}
\PrimeFilter(X)
& \cong &
\Hom(X, \{0,1\}).
\end{array}$$

\noindent (Here we consider $\{0,1\}$ as meet semilattice, with the
$S$-semimodule, and hence convex, structure described
in~\eqref{MSLsemimoduleEqn}.)
\end{lemma}

\begin{proof}
Let $\alpha\colon\Dstr_{S}(X)\rightarrow X$ be an algebra on $X$.
Given a prime filter $U\subseteq X$, define $f_{U}(x) = 1$ iff $x\in
U$. This yields a homormophism of algebras/convex sets, since for a
convex sum $\sum_{i}s_{i}x_{i}$ with $s_{i}\neq 0$,
$$\begin{array}{rcl}
(f_{U}\after\alpha)(\sum_{i}s_{i}x_{i}) = 1
& \Longleftrightarrow &
\alpha(\sum_{i}s_{i}x_{i}) \in U \\
& \Longleftrightarrow &
\all{i}{x_{i}\in U} \qquad\mbox{since $U$ is a prime filter} \\
& \Longleftrightarrow &
\all{i}{f_{U}(x_{i}) = 1} \\
& \Longleftrightarrow &
\sum_{i}s_{i}f_{U}(x) = \bigwedge_{i}f_{U}(x_{i}) = 1 \qquad
   \mbox{as in~\eqref{MSLsemimoduleEqn}} \\
& \Longleftrightarrow &
(\beta \after \Dstr_{S}(f_{U}))(\sum_{i}s_{i}x_{i}) = 1,
\end{array}$$

\noindent where $\beta\colon\Dstr_{S}(\{0,1\})\rightarrow \{0,1\}$ is
the convex structure induced by the meet semilattice structure of
$\{0,1\}$. Similarly one shows that
such homomorphisms induce prime filters as their true-kernels. \QED

\auxproof{
Given a homomorphism $f\colon X\rightarrow \{0,1\}$ the subset
$U_{f} = f^{-1}(1)$ is a prime filter since for a
convex sum $\sum_{i}s_{i}x_{i}$ with $s_{i}\neq 0$,
$$\begin{array}{rcl}
\alpha(\sum_{i}s_{i}x_{i})\in U_{f}
& \Longleftrightarrow &
f(\alpha(\sum_{i}s_{i}x_{i})) = \bigwedge_{i}f(x_{i}) = 1 \\
& \Longleftrightarrow &
\all{i}{f(x_{i}) = 1} \\
& \Longleftrightarrow &
\all{i}{x_{i}\in U_{f}}.
\end{array}$$

As to the order isomorphism, for $U,V\in\PrimeFilter(X)$,
$$\begin{array}{rcl}
U \subseteq V
& \Longleftrightarrow &
\all{x}{x\in U \Rightarrow x\in V} \\
& \Longleftrightarrow &
\all{x}{f_{U}(x) = 1 \Rightarrow f_{V}(x) = 1} \\
& \Longleftrightarrow &
f_{U} \leq f_{V}.
\end{array}$$
}
\end{proof}

We write \Cat{PreFrm} for the category of preframes. They consist of a
poset $L$ with directed joins $\dirjoin$ and finite meets $(1,\wedge)$
distributing over these joins: $x\conjun \dirjoin_{i}y_{i} =
\dirjoin_{i}x\conjun y_{i}$. Morphisms in \Cat{PreFrm} preserve both
finite meets and directed joins. The two-element set $\{0,1\}$ is
obviously a preframe. Homomorphisms of preframes $L\rightarrow\{0,1\}$
correspond (as true-kernels) to Scott-open filters $U\subseteq L$,
see~\cite{Vickers89}.  They are upsets, closed under finite meets,
with the property that if $\dirjoin_{i}x_{i}\in U$ then $x_{i}\in U$
for some $i$.

We have seen so far that taking prime filters yields a contravariant
functor $\PrimeFilter = \Hom(-,\{0,1\}) \colon \Cat{Conv}_{S} =
\Alg(\Dstr_{S}) \rightarrow \Cat{PreFrm}$. The main result of this
section shows that this forms actually a (dual) adjunction.

\begin{theorem}
\label{ConvPreFrmDualAdjThm}
For each non-trivial zerosumfree and integral semiring $S$ there is a dual
adjunction between $S$-convex sets and preframes:
$$\xymatrix{
\big(\Cat{Conv}_{S}\big)\op \ar@/_2ex/ [rr]_{\Hom(-,\{0,1\})} & \bot & 
   \Cat{PreFrm}\ar@/_2ex/[ll]_{\Hom(-,\{0,1\})}
}$$
\end{theorem}

\begin{proof}
For a preframe $L$ the homset $\Hom(L,\{0,1\})$ of Scott-open filters
is closed under finite intersections: if $\dirjoin_{i}x_{i} \in
U_{1}\cap \cdots\cap U_{m}$, then for each $j\leq m$ there is an
$i_{j}$ with $x_{j}\in U_{i_j}$. By directedness there is an $i$ with
$x_{i} \geq x_{i_j}$ for each $j$, so that $x_{i}$ is in each $U_{j}$.
Hence, $\Hom(L,\{0,1\})$ carries a $\Dstr_{S}$-algebra
structure as in~\eqref{MSLsemimoduleEqn}. We shall write it
as $\beta\colon \Dstr_{S}(\Hom(L,\{0,1\})) \rightarrow \Hom(L,\{0,1\})$.

For an $S$-convex set $X$ we need to construct a bijective
correspondence:
$$\begin{prooftree}
{\xymatrix{X \ar[r]^-{f} & \Hom(L,\{0,1\})}}
   \rlap{\qquad in $\Cat{Conv}_{S}$} 
\Justifies
{\xymatrix{L \ar[r]_-{g} & \Hom(X,\{0,1\})}}
   \rlap{\qquad in \Cat{PreFrm}}
\end{prooftree}$$

\noindent The correspondence between these $f$ and $g$ is given in
the usual way by swapping arguments. \QED

\auxproof{
Given $f$ we obtain $\overline{f}\colon L\rightarrow
\Hom(X,\{0,1\})$ by $\overline{f}(a)(x) = f(x)(a)$. We need to check that
it is well-defined. First, for each $a\in L$, $\overline{f}(a)$ is a
homomorphism of algebras/convex sets $X\rightarrow \{0,1\}$, since:
$$\begin{array}{rcl}
\overline{f}(a)(\alpha(\sum_{i}s_{i}x_{i}))
\hspace*{\arraycolsep}=\hspace*{\arraycolsep}
f(\alpha(\sum_{i}s_{i}x_{i}))(a) 
& = &
\beta(\Dstr_{S}(f)(\sum_{i}s_{i}x_{i}))(a) \\
& = &
(\bigwedge_{i}f(x_{i}))(a) \\
& = &
\bigwedge_{i}f(x_{i})(a) \\
& = &
\bigwedge_{i}\overline{f}(a)(x_{i}) \\
& = &
\sum_{i}s_{i}(\overline{f}(a)(x_{i})).
\end{array}$$

\noindent Next, $\overline{f}$ must preserve finite meets and
directed joins. This is easy by pointwise reasoning:
$$\textstyle\overline{f}(\bigwedge_{i} a_{i})(x)
=
f(x)(\bigwedge_{i}a_{i})
=
\bigwedge_{i}f(x)(a_{i}) 
=
\bigwedge_{i}\overline{f}(a_{i})(x) 
=
(\bigwedge_{i}\overline{f}(a_{i}))(x).$$

In the reverse direction $g\colon L\rightarrow \Hom(X,\{0,1\})$ yields
$\overline{g}\colon X\rightarrow \Hom(L,\{0,1\})$ by twisting
arguments, as in $\overline{g}(x)(a) = g(a)(x)$. This $\overline{g}(x)$
is a homomorphism of preframes, since:
$$\textstyle\overline{g}(x)(\bigwedge_{i}a_{i})
=
g(\bigwedge_{i}a_{i})(x)
=
(\bigwedge_{i}g(a_{i}))(x)
=
\bigwedge_{i}g(a_{i})(x) 
=
\bigwedge_{i}\overline{g}(x)(a_{i}).$$

\noindent And $\overline{g}$ is a map of algebras/convex sets:
$$\begin{array}{rcl}
\overline{g}(\alpha(\sum_{i}s_{i}x_{i}))
\hspace*{\arraycolsep}=\hspace*{\arraycolsep}
\lam{a}{\overline{g}(\alpha(\sum_{i}s_{i}x_{i}))(a)} 
& = &
\lam{a}{g(a)(\alpha(\sum_{i}s_{i}x_{i}))} \\
& = &
\lam{a}{\sum_{i}s_{i}g(a)(x_{i})} \\
& = &
\lam{a}{\bigwedge_{i}g(a)(x_{i})} \\
& = &
\lam{a}{\bigwedge_{i}\overline{g}(x_{i})(a)} \\
& = &
\lam{a}{(\bigwedge_{i}\overline{g}(x_{i}))(a)} \\
& = &
\bigwedge_{i}\overline{g}(x_{i}) \\
& = &
\sum_{i}s_{i}\overline{g}(x_{i}).
\end{array}$$
}
\end{proof}

Homomorphisms from convex sets to the set of Boolean values $\{0,1\}$
capture only a part of what is going on. Richer structures arise via
homomorphisms to the unit interval $\unitR$. They give rise to effect
algebras, instead of preframes, as will be shown in the next two
sections.

\section{Effect algebras}\label{EffAlgSec}

This section recalls the basic definition, examples and result of
effect algebras. To start, we need the notion of partial commutative
monoid. It consists of a set $M$ with a zero element $0\in M$ and a
partial binary operation $\ovee\colon M\times M\rightarrow M$
satisfying the three requirements below---involving the notation
$x\orthogonal y$ for: $x\ovee y$ is defined.
\begin{enumerate}
\item Commutativity: $x\orthogonal y$ implies $y\orthogonal x$ and
$x\ovee y = y\ovee x$;

\item Associativity: $y\orthogonal z$ and $x \orthogonal (y\ovee z)$
implies $x\orthogonal y$ and $(x\ovee y) \orthogonal z$ and also
$x \ovee (y\ovee z) = (x\ovee y)\ovee z$;

\item Zero: $0\orthogonal x$ and $0\ovee x = x$;
\end{enumerate}

An example of a partially commutative monoid is the unit interval
$\unitR$ of real numbers, where $\ovee$ is the partially defined sum
$+$.  The notation $\ovee$ for the sum might suggest a join, but this
is not intended, as the example $\unitR$ shows. We wish to avoid the
notation $\oplus$ (and its dual $\otimes$) that is more common in the
context of effect algebras because we like to reserve these operations
$\oplus, \otimes$ for tensors on categories.

As an aside, for the more categorically minded, a partial commutative
monoid may also be understood as a monoid in the category
$\Sets_\bullet$ of pointed sets (or sets and partial
functions). However, morphisms of partially commutative monoids are
mostly total maps.

\auxproof{
It then consists of a pointed set $M\in\Sets_\bullet$
with two morphisms $2 \stackrel{e}{\rightarrow} M
\stackrel{m}{\leftarrow} M\otimes M$ in the category $\Sets_{\bullet}$
of pointed sets making the monoid diagrams commute. In particular, 
the composite
$$\xymatrix{
M \ar[r]^-{\cong} & 2\otimes M \ar[r]^-{e\otimes\idmap} & 
   M\otimes M\ar[r]^-{m} & M
}$$

\noindent is the identity $M\rightarrow M$. This means that $e(1)\neq
0_{M}$, where $1\in 2 = \{0,1\}$.
}

The notion of effect algebra is due to~\cite{FoulisB94}, see
also~\cite{DvurecenskijP00} for an overview.

\begin{definition}
\label{EffAlgDef}
An effect algebra is a partial commutative monoid $(E, 0, \ovee)$ with
an orthosupplement. The latter is a unary operation $(-)^{\perp}
\colon E\rightarrow E$ satisfying:
\begin{enumerate}
\item $x^{\perp}\in E$ is the unique element in $E$ with $x\ovee
  x^{\perp} = 1$, where $1 = 0^\perp$;

\item $x\orthogonal 1 \Rightarrow x=0$.
\end{enumerate}


\end{definition}

When writing $x\ovee y$ we shall implicitly assume that $x\ovee y$ is
defined, \textit{i.e.} that $x\orthogonal y$ holds. 

\begin{example}
\label{EffAlgEx}
We briefly discuss several classes of examples.  {\em(1)}~A singleton
set forms an example of a degenerate effect algebra, with $0=1$. A two
element set $2 = \{0,1\}$ is also an example.

{\em(2)}~A more interesting example is the unit interval $\unitR
\subseteq \mathbb{R}$ of real numbers, with $r^{\perp} = 1-r$ and
$r\ovee s$ is defined as $r+s$ in case this sum is in $\unitR$. In
fact, for each positive number $M\in\mathbb{R}$ the interval
$[0,M]_{\mathbb{R}} = \setin{r}{\mathbb{R}}{0\leq r \leq M}$ is an
example of an effect algebra, with $r^{\perp} = M-r$.

Also the interval $[0,M]_{\mathbb{Q}} = \setin{q}{\mathbb{Q}}{0 \leq q
  \leq M}$ of rational numbers, for positive $M\in\mathbb{Q}$, is an
effect algebra. And so is the interval $[0,M]_{\NNO}$ of natural
numbers, for $M\in\NNO$.

The general situation involves so-called ``interval effect algebras'',
see \textit{e.g.}~\cite{FoulisGB98} or~\cite[1.4]{DvurecenskijP00}.  An
Abelian group $(G,0,-, +)$ is called ordered if it carries a partial
order $\leq$ such that $a\leq b$ implies $a+c \leq b+c$, for all
$a,b,c\in G$. A positive point is an element $p\in G$ with $p\geq
0$. For such a point we write $[0,p]_{G}\subseteq G$ for the
``interval'' $[0,p] = \setin{a}{G}{0 \leq a \leq p}$. It forms an
effect algebra with $p$ as top, orthosupplement $a^{\perp} = p-a$, and
sum $a+b$, which is considered to be defined in case $a+b\leq p$. 

\auxproof{
In an ordered Abelian group,
\begin{enumerate}
\item $a\leq b$ implies $-b \leq -a$; 

\item $a\leq b$ iff $\ex{c\geq 0}{a+c = b}$;
\end{enumerate}

The first point is easy: if $a\leq b$, then $-b = a + (-a-b) \leq b +
(-a-b) = -a$. For the second point, notice that if $a \leq b$, then $b
= a + c$ for $c = (b-a)$, where $0 = a-a \leq b-a = c$. Conversely, if
$a+c = b$ for $c\geq 0$, then $-c \leq -0 = 0$ by~(1), and thus $a =
-c + b \leq 0+ b = b$. 
}

{\em(3)}~A separate class of examples has a join as sum $\ovee$.
Let $(L, \vee, 0, (-)^{\perp})$ be an ortholattice: $\vee, 0$ are
finite joins and complementation $(-)^{\perp}$ satisfies $x\leq y
\Rightarrow y^{\perp} \leq x^{\perp}$, $x^{\perp\perp} = x$ and
$x\disjun x^{\perp} = 1 = 0^{\perp}$. This $L$ is called an
orthomodular lattice if $x \leq y$ implies $y = x \disjun (x^{\perp}
\conjun y)$. Such an orthomodular lattice forms an effect algebra in
which $x\ovee y$ is defined if and only if $x\orthogonal y$
(\textit{i.e.}~$x\leq y^{\perp}$, or equivalently, $y \leq
x^{\perp}$); and in that case $x\ovee y = x\disjun y$. This
restriction of $\vee$ is needed for the validity of requirements~(1)
and~(2) in Definition~\ref{EffAlgDef}: 
\begin{itemize}
\item suppose $x\ovee y = 1$, where $x\orthogonal y$,
  \textit{i.e.}~$x\leq y^{\perp}$. Then, by the orthomodularity
  property,
$$y^{\perp}
=
x \disjun (x^{\perp} \conjun y^{\perp})
=
x\disjun (x\disjun y)^{\perp}
=
x \disjun 1^{\perp}
=
x\disjun 0
=
x.$$

Hence $y = x^{\perp}$, making orthosupplements unique.

\item $x\orthogonal 1$ means $x \leq 1^{\perp} = 0$, so that $x=0$.
\end{itemize}

\noindent In particular, the lattice $\KSub(H)$ of closed subsets of a
Hilbert space $H$ is an orthomodular lattice and thus an effect
algebra. This applies more generally to the kernel subobjects
of an object in a dagger kernel category~\cite{HeunenJ10a}. These
kernels can also be described as self-adjoint endomaps below the
identity, see~\cite[Prop.~12]{HeunenJ10a}---in group-representation
style, like in the above point~2.

{\em(4)}~Since Boolean algebras are (distributive) orthomodular
lattices, they are also effect algebras. By distributivity, elements
in a Boolean algebra are orthogonal if and only if they are disjoint,
\textit{i.e.}~$x\orthogonal y$ iff $x\conjun y = 0$. In particular,
the Boolean algebra of measurable subsets of a measurable space forms
an effect algebra, where $U\ovee V$ is defined if $U\cap V =
\emptyset$, and is then equal to $U\cup V$.

\auxproof{
If $x\orthogonal y$, then $x\leq \neg y$, so that $x\conjun y \leq
\neg y \conjun y = 0$. Conversely, if $x\conjun y = 0$, then by
distributivity $x = x \conjun (y\disjun \neg y) = (x\conjun y)
\disjun (x\conjun \neg y) = x\conjun \neg y$. Hence $x\leq \neg y$.
}
\end{example}

An obvious next step is to organise effect algebras into a category
\Cat{EA}.

\begin{definition}
\label{EffAlgHomDef}
A homomorphism $E\rightarrow D$ of effect algebras is given by a function
$f\colon E\rightarrow D$ between the underlying sets satisfying:
\begin{itemize}
\item $x\orthogonal x'$ in $E$ implies both $f(x) \orthogonal f(x')$
  in $D$ and $f(x\ovee x') = f(x) \ovee f(x')$;

\item $f(1) = 1$.
\end{itemize}

\noindent Effect algebras and their homomorphisms form a category,
which we call \Cat{EA}.
\end{definition}

Homomorphisms are like measurable maps. Indeed, for the effect algebra
$\Sigma$ associated in Example~\ref{EffAlgEx}~(4) with a measureable
space $(X,\Sigma)$, effect algebra homomorphisms $f\colon \Sigma
\rightarrow \unitR$ satisfy $f(U\cup V) = f(U)+f(V)$ in case $U,V$ are
disjoint---because then $U\ovee V$ is defined and equals $U\cup V$. In
general, effect algebra homomorphisms $E\rightarrow \unitR$ to the
unit interval are often called states. They form a convex subset, see
Section~\ref{EffAlgConvexSec}.

Homomorphisms of effect algebras preserve all the relevant structure.



\begin{lemma}
\label{EffAlgHomLem}
Let $f\colon E\rightarrow D$ be a homomorphism of effect algebras. Then:
$$f(x^{\perp}) = f(x)^{\perp}
\quad\mbox{and thus}\quad
f(0) = 0.$$


\end{lemma}

\begin{proof}
From $1 = f(1) = f(x \ovee x^{\perp}) = f(x) \ovee f(x^{\perp})$ we obtain
$f(x^{\perp}) = f(x)^{\perp}$ by uniqueness of orthosupplements. In
particular, $f(0) = f(1^{\perp}) = f(1)^{\perp} = 1^{\perp} = 0$. \QED
\end{proof}

\begin{example}
\label{EffAlgHomEx}
It is not hard to see that the one-element effect algebra $1$ is
final, and the two-element effect algebra $2$ is initial.

Orthosupplement $(-)^{\perp}$ is a homomorphism $E\rightarrow E\op$ in
\Cat{EA}, namely from $(E, 0, \ovee, (-)^{\perp})$ to $E\op = (E, 1,
\owedge, (-)^{\perp})$, where $x\owedge y = (x^{\perp}\ovee
y^{\perp})^{\perp}$.

An element (or point) $x\in E$ of an effect algebra $E$ can be
identified with a homomorphism $2\times 2\rightarrow E$ in $\Cat{EA}$, as in:
$$\xymatrix{
2\times 2 = \mathsf{MO}(2) = \ensuremath{\left(\xy
(0,4)*{1};
(-4,0)*{\bullet}; 
(4,1)*{\bullet\rlap{$^\perp$}};
(0,-4)*{0};
{\ar@{-} (-1,3); (-3,1)};
{\ar@{-} (1,3); (3,1)};
{\ar@{-} (-1,-3); (-3,-1)};
{\ar@{-} (1,-3); (3,-1)};
\endxy\;\;\right)}
\ar[rr]^-{x} & & E
}$$

On the homset $\Hom(E,D)$ of homomorphisms $E\rightarrow D$ in
\Cat{EA} one may define $(-)^{\perp}$ and $\ovee$ pointwise, as in
$f^{\perp}(x) = f(x)^{\perp}$. But this does not yield a homomorphism
$E\rightarrow D$, since for instance $f^{\perp}(1) = f(1)^{\perp} =
1^{\perp} = 0$. Hence these homsets do not form effect algebras.


\end{example}

\begin{example}
\label{RationalUnitStateEx}
Recall from Example~\ref{EffAlgEx}.(2) the effect algebra
$[0,1]_{\mathbb{Q}}$ given by the unit interval of rational numbers.
We claim that it has precisely one state: there is precisely one
morphism of effect algebras $f\colon [0,1]_{\mathbb{Q}} \rightarrow
\unitR$, namely the inclusion. To see this we first prove that
$f(\frac{1}{n}) = \frac{1}{n}$ for each positive $n\in\NNO$. Since the
$n$-fold sum $\frac{1}{n} + \cdots + \frac{1}{n}$ equals $1$ this
follows from:
$$\textstyle 1 
=
f(1)
=
f(\frac{1}{n} + \cdots + \frac{1}{n})
=
f(\frac{1}{n}) + \cdots + f(\frac{1}{n}).$$

\noindent Similarly we get $f(\frac{m}{n}) = \frac{m}{n}$, for $m\leq n$,
via an $m$-fold sum:
$$\textstyle f(\frac{m}{n})
=
f(\frac{1}{n} + \cdots + \frac{1}{n})
=
f(\frac{1}{n}) + \cdots + f(\frac{1}{n})
=
\frac{1}{n} + \cdots + \frac{1}{n} 
=
\frac{m}{n}.$$
\end{example}

We briefly mention some basic structure in the category of effect
algebras.

\begin{proposition}
\label{EffAlgCompletenessProp}
The category \Cat{EA} of effect algebras is complete, where products and
equalisers are constructed as in $\Sets$ and equipped with the appropriate
effect algebra structure.

The category \Cat{EA} also has set-indexed coproducts, given by
identifying top and bottom elements, as in: the coproduct $E+D =
\big((E - \{0,1\}) + (D-\{0,1\})\big) + \{0,1\}$, where $+$ on the
right-hand-side of the equality is disjoint union (or coproduct) of
sets. \QED
\end{proposition}

Coequalisers in \Cat{EA} are more complicated, but are not needed
here.

\section{Effect algebras and convex sets}\label{EffAlgConvexSec}

Our aim in this section is to establish the dual adjunction between
convex sets and effect algebras on the right in the
diagram~\eqref{DualAdjsDiag} in the introduction.  From now on we
restrict ourselves to the semiring $\nonnegR$ of non-negative real
numbers.  As before, we omit it as subscript and write $\Dstr$ for
$\Dstr_{\nonnegR}$ and $\Cat{Conv}$ for $\Cat{Conv}_{\nonnegR} =
\Alg(\Dstr_{\nonnegR})$.

We already mentioned that the unit interval $\unitR$ of real numbers
is a convex set. The set of states of an effect algebra is also
convex, as noticed for instance in~\cite{FoulisGB98}.

\begin{lemma}
\label{EffAlgToConvLem}
The state functor $\mathcal{S} = \Hom(-, \unitR) \colon
\Cat{EA} \rightarrow \Sets\op$ restricts to $\Cat{EA} \rightarrow
\Cat{Conv}\op$.
\end{lemma}

\begin{proof}
Let $E$ be an effect algebra with states $f_{i}\colon E\rightarrow
\unitR$ and $r_{i}\in \unitR$ with $\sum_{i}r_{i} = 1$, then we can
form a new state $f = r_{1}f_{1} + \cdots + r_{n}f_{n}$ by $f(x) =
\sum_{i}r_{i}\cdot f_{i}(x)$, using multiplication $\cdot$ in
$\unitR$. This yields a homomorphism of effect algebras $E\rightarrow
\unitR$, since:
\begin{itemize}
\item $f(1) = \sum_{i}r_{i}\cdot f_{i}(1) = \sum_{i}r_{i}\cdot 1 = 
\sum_{i}r_{i} = 1$;

\item if $x\orthogonal x'$ in $E$, then in $\unitR$:
$$\begin{array}{rcl}
f(x\ovee x')
\hspace*{\arraycolsep}=\hspace*{\arraycolsep}
\sum_{i}r_{i}\cdot f_{i}(x\ovee x')
& = &
\sum_{i}r_{i}\cdot (f_{i}(x) + f_{i}(x')) \\
& = &
\sum_{i}r_{i}\cdot f_{i}(x) + r_{i}\cdot f_{i}(x') \\
& = &
\sum_{i}r_{i}\cdot f_{i}(x) + \sum_{i}r_{i}\cdot f_{i}(x') \\
& = &
f(x) + f(x').
\end{array}$$
\end{itemize}

\noindent Further, for a map of effect algebras $g\colon E\rightarrow
D$ the induced function $\mathcal{S}(g) = (-) \after g \colon
\Hom(D,\unitR) \rightarrow \Hom(E, \unitR)$ is a map of convex sets:
$$\begin{array}[b]{rcl}
\mathcal{S}(g)(\sum_{i}r_{i}f_{i})
& = &
\lam{x}{(\sum_{i}r_{i}f_{i})(g(x))} \\
& = &
\lam{x}{\sum_{i}r_{i}\cdot f_{i}(g(x))} \\
& = &
\lam{x}{\sum_{i}r_{i}\cdot \mathcal{S}(g)(f_{i})(x)} \\
& = &
\sum_{i}r_{i}(\mathcal{S}(g)(f_{i})).
\end{array}\eqno{\QEDbox}$$
\end{proof}

Interestingly, there is also a Hom functor in the other direction.

\begin{lemma}
\label{ConvToEffAlgLem}
For each convex set $X$ the homset $\Hom(X,\unitR)$ of homomorphisms
of convex sets is an effect algebra. In this way one gets a functor
$\Hom(-,\unitR) \colon \Cat{Conv}\op \rightarrow \Cat{EA}$.
\end{lemma}

\begin{proof}
Let $X$ be a convex set. We have to define effect algebra structure on
the homset $\Hom(X,\unitR)$. There is an obvious zero element, namely
the zero function $\lam{x}{0}$. A partial sum $f+f'$ is defined as
$(f+f')(x) = f(x) + f'(x)$, provided the sum $f(x)+f'(x)\leq 1$ for
all $x\in X$. It is easy to see that this $f+f'$ is again a map of
convex sets. Similarly, one defines $f^{\perp} = \lam{x}{1-f(x)}$, which
is again a homomorphism since:
$$\begin{array}{rcl}
f^{\perp}(r_{1}x_{1} + \cdots + r_{n}x_{n})
& = &
1 - f(r_{1}x_{1} + \cdots + r_{n}x_{n}) \\
& = &
(r_{1} + \cdots + r_{n}) - (r_{1}\cdot f(x_{1}) + \cdots + r_{n}\cdot f(x_{n})) \\
& = &
r_{1}\cdot (1-f(x_{1})) + \cdots + r_{n}\cdot (1-f(x_{n})) \\
& = &
r_{1}\cdot f^{\perp}(x_{1}) + \cdots + r_{n}\cdot f^{\perp}(x_{n}).
\end{array}$$

\auxproof{
$$\begin{array}{rcl}
\lefteqn{(f+f')(r_{1}x_{1} + \cdots + r_{n}x_{n})} \\
& = &
f(r_{1}x_{1} + \cdots + r_{n}x_{n}) + f'(r_{1}x_{1} + \cdots + r_{n}x_{n}) \\
& = &
r_{1}\cdot f(x_{1}) + \cdots + r_{n}\cdot f(x_{n}) + 
   r_{1}\cdot f'(x_{1}) + \cdots + r_{n}\cdot f'(x_{n}) \\
& = &
r_{1}\cdot (f(x_{1}) + f'(x_{1})) + \cdots + r_{n}\cdot (f(x_{n}) + f'(x_{n})) \\
& = &
r_{1}\cdot (f+f')(x_{1}) + \cdots + r_{n}\cdot (f+f')(x_{n})
\end{array}$$
}

\noindent Functoriality is easy: for a map $g\colon X\rightarrow Y$ of
convex sets we obtain a map of effect algebras 
$(-)\after g\colon \Hom(Y,\unitR) \rightarrow \Hom(X,\unitR)$ since:
\begin{itemize}
\item $1 \after g = \lam{x}{1(g(x))} = \lam{x}{1} = 1$;

\item $(f+f') \after g = \lam{x}{(f+f')(g(x))} = \lam{x}{f(g(x)) + f'(g(x))}
= \lam{x}{(f \after g)(x) + (f'\after g)(x)} = (f\after g) + (f'\after g)$.
\QED
\end{itemize}
\end{proof}

The next result is now an easy combination of the previous two lemmas.

\begin{theorem}
\label{ConvEffAlgDualAdjThm}
There is a dual adjunction between convex sets and effect algebras:
$$\xymatrix{
\Cat{Conv}\op \ar@/_2ex/ [rr]_{\Hom(-,\unitR)} & \bot & 
   \Cat{EA}\ar@/_2ex/[ll]_{\mathcal{S} = \Hom(-,\unitR)}
}$$
\end{theorem}

\begin{proof}
We need to check that the unit and counit
$$\xymatrix@R-1.8pc{
E\ar[r]^-{\eta} & \Hom(\mathcal{S}(E), \unitR)
&
X\ar[r]^-{\varepsilon} & \mathcal{S}(\Hom(X,\unitR)) \\
x\;\ar@{|->}[r] & \lam{f}{f(x)}
&
x\;\ar@{|->}[r] & \lam{f}{f(x)}
}$$

\noindent are appropriate homomorphisms. First we check that $\eta$
is a map of effect algebras:
\begin{itemize}
\item $\eta(1) = \lam{f}{f(1)} = \lam{f}{1} = 1$;

\item and if $x\orthogonal x'$ in $E$, then:
$$\begin{array}{rcl}
\eta(x\ovee x')
\hspace*{\arraycolsep}=\hspace*{\arraycolsep}
\lam{f}{f(x\ovee x')}
& = &
\lam{f}{f(x) + f(x')} \\
& = &
\lam{f}{\eta(x)(f) + \eta(x')(f)} \\
& = &
\eta(x) + \eta(x').
\end{array}$$
\end{itemize}

\noindent Similarly $\varepsilon$ is a map of convex sets:
$$\begin{array}[b]{rcl}
\varepsilon(r_{1}x_{1} + \cdots + r_{n}x_{n})
& = &
\lam{f}{f(r_{1}x_{1} + \cdots + r_{n}x_{n})} \\
& = &
\lam{f}{r_{1}\cdot f(x_{1}) + \cdots + r_{n}\cdot f(x_{n})} \\
& = &
\lam{f}{r_{1}\cdot \varepsilon(x_{1})(f) + \cdots + 
   r_{n}\cdot \varepsilon(x_{n})(f)} \\
& = &
r_{1}\varepsilon(x_{1}) + \cdots + r_{n}\varepsilon(x_{n}).
\end{array}\eqno{\QEDbox}$$
\end{proof}

Recall that the forgetful functor $U\colon \Cat{Conv}
=\Alg(\Dstr)\rightarrow \Sets$ has a left adjoint, also written as
$\Dstr$. By taking opposites $\Dstr$ becomes a right adjoint to $U$,
so that we can compose adjoint, as in the following result.

\begin{proposition}
\label{ConvEffAlgDualAdjCompProp}
By composition of adjoints, as in:
$$
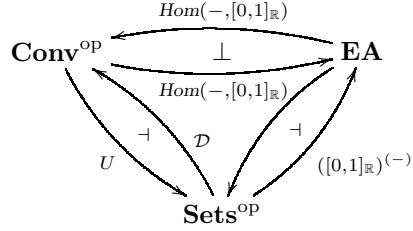
\xymatrix@R+2pc{
\Cat{Conv}\op \ar@/_2ex/ [rr]_-{\Hom(-,\unitR)}\ar@/_2ex/[dr]_{U}^{\;\dashv}
   & \bot & 
   \Cat{EA}\ar@/_2ex/[ll]_-{\Hom(-,\unitR)}\ar@/_2ex/[dl]^{\;\dashv} \\
& \Sets\op\ar@/_2ex/[ul]_(0.3){\Dstr}\ar@/_2ex/[ur]_{(\unitR)^{(-)}}
}$$

\noindent one obtains in a standard way a dual adjunction between
effect algebras and sets.
\end{proposition}

\begin{proof}
Because by the adjunction $\Dstr \dashv U$, for $X\in\Sets$, 
$$\Hom_{\Cat{Conv}}(\Dstr(X), \unitR)
\cong
\Hom_{\Sets}(X, U(\unitR))
\cong
\big(\unitR\big)^{X}$$

\noindent which, yields an effect algebra because by
Proposition~\ref{EffAlgCompletenessProp} effect algebras are closed
under products (and hence under powers). Its right adjoint is $E
\mapsto \Hom_{\Cat{EA}}(E, \unitR) = U\big(\Hom_{\Cat{EA}}(E,
\unitR)\big)$, where this last homset is considered as object of the
category \Cat{Conv}. \QED
\end{proof}

\section{Hilbert spaces}\label{HilbertSec}

In the end one may ask: how does the standard way of modeling quantum
phenomena in Hilbert spaces fit in the picture~\eqref{DualAdjsDiag}
provided by the dual adjunctions? The answer is: only partially. There
is a contravariant functor from Hilbert spaces to effect algebras,
mapping a Hilbert space to its orthomodular lattice (and hence effect
algebra) of closed subspaces. The unit ball $H_1$ in each Hilbert
space $H$---with $H_1$ consisting of points $a\in H$ with $\|a\|\leq
1$---is convex. However, this mapping $H\mapsto H_1$ is not functorial
in an obvious way. Nevertheless, each unit element does give rise to
a state, as described in the following lemma. It uses notation and
terminology from~\cite{HeunenJ10a,Heunen10a}.

\begin{lemma}
\label{HilbertCounitLem}
Let $H$ be a Hilbert space, with a unit element $a\in H$ (so
that $\|a\|=1$). It gives rise to a map of effect algebras:
$$\xymatrix{
\KSub(H)\ar[r]^-{\varepsilon(a)} & 
\unitR, \quad\mbox{namely}\quad
k\, \ar@{|->}[r] & \inprod{k^{\dag}(a)}{k^{\dag}(a)} = \|k^{\dag}(a)\|^{2}
}$$

\noindent where $\KSub(H)$ is the orthomodular lattice of
closed subspaces $k\colon K\rightarrowtail H$.
\end{lemma}

\begin{proof}
  Clearly $\varepsilon(a)(k) \geq 0$. Further, $\varepsilon(a)$
  preserves the top element:
$$\varepsilon(a)(1)
=
\varepsilon(a)(\idmap[H])
=
\|\idmap^{\dag}(a)\|^{2}
=
\|\idmap(a)\|^{2}
=
\|a\|^{2}
= 
1^{2}
=
1.$$

\noindent We show that $\varepsilon(a)$ is monotone. Assume therefor
$k\leq k'$ in $\KSub(H)$. Then we can write $k'$ as cotuple
$\cotuple{k, m}\colon K\oplus M\rightarrowtail H$, where $K\oplus M =
K\times M$ describes the biproduct in \Hilb, so that:
$$\begin{array}{rcl}
\varepsilon(a)(k')
\hspace*{\arraycolsep}=\hspace*{\arraycolsep}
\varepsilon(a)([k,m])
& = &
\inprod{\cotuple{k, m}^{\dag}(a)}{\cotuple{k, m}^{\dag}(a)} \\
& = &
\inprod{\tuple{k^{\dag}, m^{\dag}}(a)}{\tuple{k^{\dag}, m^{\dag}}(a)} \\
& = &
\inprod{\tuple{k^{\dag}(a), m^{\dag}(a)}}{\tuple{k^{\dag}(a), m^{\dag}(a)}} \\
& = &
\inprod{k^{\dag}(a)}{k^{\dag}(a)} + \inprod{m^{\dag}(a)}{m^{\dag}(a)} \\
& = &
\varepsilon(a)(k) + \varepsilon(a)(m).
\end{array}$$

\noindent Hence $\varepsilon(a)(k) \leq \varepsilon(a)(k')$ in
$\mathbb{R}$.  In particular, since each $k\in\KSub(H)$ satisfies
$k\leq 1$ we get $\varepsilon(a)(k) \leq \varepsilon(a)(1) =
1$. Therefor $\varepsilon(a)(k) \in \unitR$.

Finally we show that $\varepsilon(a)$ is a map of effect algebras.
Assume $k\orthogonal m$ for $k,m\in\KSub(H)$. This means $k\leq
m^{\dag}$ so that $k^{\dag} \after m = 0$ and also $m^{\dag} \after k
= 0$. Hence the cotuple $\cotuple{k, m}$ is a dagger mono, and thus
the join $k\ovee m$. Then, like in the previous computation:
$$\varepsilon(a)(k\ovee m)
=
\varepsilon(a)(\cotuple{k, m})
=
\varepsilon(a)(k) + \varepsilon(a)(m).\eqno{\QEDbox}$$

\auxproof{
The cotuple $\cotuple{k, m}$ is indeed a dagger mono:
$$\begin{array}{rcl}
\cotuple{k, m}^{\dag} \after \cotuple{k, m}
& = &
\tuple{k^{\dag}, m^{\dag}} \after \cotuple{k, m} \\
& = &
\tuple{k^{\dag}\after \cotuple{k, m}, m^{\dag}\after \cotuple{k, m}}  \\
& = &
\tuple{\cotuple{k^{\dag}\after k, k^{\dag}\after m}, 
   \cotuple{m^{\dag}\after k, m^{\dag}\after m}}  \\
& = &
\tuple{\cotuple{\idmap, 0}, \cotuple{0, \idmap}}  \\
& = &
\tuple{\pi_{1}, \pi_{2}} \\
& = &
\idmap.
\end{array}$$
}
\end{proof}

The resulting mapping $\varepsilon\colon H_{1} \rightarrow
\mathcal{S}(\KSub(H)) = \Cat{EA}(\KSub(H), \unitR)$ need not preserve
convex sums. Hence Hilbert spaces do not fit nicely in the dual
adjunctions diagram~\eqref{DualAdjsDiag}. More research is needed to
clarify the situation. In particular, it seems worthwhile to bring
compact and Hausdorff spaces into the picture, like
in~\cite{Keimel08}, and to look for restrictions of the dual
adjunction in Theorem~\ref{ConvEffAlgDualAdjThm}, possibly involving
$C^{*}$-algebras (instead of Hilbert spaces).

\subsubsection*{Acknowledgements}
Thanks to Dion Coumans, Chris Heunen, Bas Spitters and Jorik
Mandemaker for feedback and/or helpful discussions.

\bibliographystyle{plain} 

\end{document}